\documentclass[11pt,oneside]{amsart}
\usepackage{amsmath,amsthm,amsfonts,amscd,amssymb,enumerate,color,hyperref,latexsym}
\usepackage[utf8]{inputenc}
\usepackage[T1,T2A]{fontenc} 
\usepackage[russian,german,english]{babel}
\input{xypic}
\usepackage{fullpage}
\usepackage{lineno}
\usepackage{rotating}
\usepackage{multirow}
\usepackage{graphicx}
\usepackage{hyperref}
\usepackage{epstopdf}
\usepackage[usenames,dvipsnames]{xcolor}
\usepackage{tikz-cd}
\usepackage{booktabs}
\usepackage{enumitem}
\usepackage{multicol}
\usepackage{etoolbox}

\usepackage{hyperref}
\hypersetup{colorlinks,allcolors=blue}
\numberwithin{equation}{section}
\numberwithin{table}{section}
\usepackage{comment}
\usepackage{geometry}
\usepackage{doi}
\usepackage{color,soul}

\theoremstyle{plain}
\newtheorem{theorem}{Theorem}[section]
\newtheorem{lemma}[theorem]{Lemma}
\newtheorem{proposition}[theorem]{Proposition}
\newtheorem{corollary}[theorem]{Corollary}
\newtheorem{claim}[theorem]{Claim}
\theoremstyle{definition}
\newtheorem{definition}[theorem]{Definition}
\theoremstyle{remark}
\newtheorem{remark}[theorem]{Remark}

\usepackage{stmaryrd}
\newcommand{\RR}{\mathbb{R}}

\newcommand{\NN}{\mathbb{N}}
\DeclareMathOperator{\diam}{diam}
\DeclareMathOperator{\vol}{vol}

\newcommand{\Id}{\mathrm{Id}}
\newcommand{\eps}{\varepsilon}
\DeclareMathOperator{\Appr}{Appr}
\DeclareMathOperator{\curv}{curv}
\newcommand{\toGH}{\xrightarrow[]{\mathsf{GH}}}
\newcommand{\GH}{\mathsf{GH}}
\newcommand{\GHPair}{\mathsf{GHPair}}
\newcommand{\GHTuple}{\mathsf{GHTuple}}
\renewcommand{\H}{\mathsf{H}}
\newcommand{\BB}{\overline{B}}
\newcommand{\calA}{\mathcal{A}}
\newcommand{\calM}{\mathcal{M}}
\newcommand{\calN}{\mathcal{N}}
\newcommand{\calS}{\mathcal{S}}

\newcommand{\calL}{\mathcal{L}}
\renewcommand{\Im}{\mathrm{Im}}

\title{Gromov--Hausdorff convergence of metric pairs and metric tuples} 
\date{\today}

\author[A.~Ahumada G\'omez]{Andr\'es Ahumada G\'omez $^\mathrm{A}$}
\address[Ahumada G\'omez]{Facultad de Ciencias, Universidad Nacional Aut\'onoma de M\'exico, Mexico}
\email{andres.ahumada.gomez@ciencias.unam.mx}
\thanks{$^{\mathrm{A}}$Supported by CONACYT Doctoral Scholarship No. 734952.}

\author[M.~Che]{Mauricio Che $^\mathrm{B}$}
\address[Che]{Department of Mathematical Sciences, Durham University, United Kingdom.}
\email{mauricio.a.che-moguel@durham.ac.uk}
\thanks{$^{\mathrm{B}}$Supported by CONACYT Doctoral Scholarship No. 769708.}


\begin{document}
	
	\maketitle

	\setcounter{tocdepth}{1}
	\tableofcontents
	
	\begin{abstract}
		We study the Gromov--Hausdorff convergence of metric pairs and metric tuples and prove the equivalence of different natural definitions of this concept. We also prove embedding, completeness and compactness theorems in this setting. Finally, we get a relative version of Fukaya's theorem about quotient spaces under Gromov--Hausdorff equivariant convergence and a version of Grove--Petersen--Wu's finiteness theorem for stratified spaces.
	\end{abstract}

	\section{Introduction}
	The Gromov--Hausdorff convergence is a way to quantitatively study metric spaces. It was introduced by M. Gromov in \cite{gromov2,gromov,gromov1999metric} (cf. \cite{edwards}). It has been used to measure the rigidity and stability of several different geometric and topological properties. Some remarkable theoretical results involving the Gromov--Hausdorff convergence include Gromov's compactness theorem \cite{gromov2,gromov,gromov1999metric} and Grove--Petersen--Wu's finiteness theorem \cite{Grove1990,Grove1991} which, in turn, have motivated the development of very fruitful theories of synthetic curvature bounds for metric spaces, namely, the study of Alexandrov spaces \cite{burago,bgp} and ${\sf RCD}$ spaces \cite{ambrosio-gigli-savare,lott-villani,sturm1,sturm2}. In a more applied spirit, the Gromov--Hausdorff convergence has also been used in the recognition of point-cloud data sets \cite{memoli05} and in Topological Data Analysis as a way to measure the stability of persistence diagrams and other topological and geometric signatures of metric spaces \cite{chazal-et-al,stability-tda}. 
	
	In this paper, we study the Gromov--Hausdorff convergence of metric pairs. This notion was introduced in \cite{CGGGMS}, where the authors studied a family of functors that produce generalised spaces of persistence diagrams and prove the continuity of the bottleneck functor under their definition of convergence for metric pairs. The Gromov--Hausdorff convergence of metric pairs is an extension of the pointed Gromov--Hausdorff convergence, which can be formulated in several different ways, such as those considered in \cite{burago,herron,jansen}. We gather the natural generalisations of these different formulations to the setting of metric pairs, prove the equivalence between the resulting definitions, and carry some well-known basic properties of the pointed case to our framework. Moreover, we prove that this notion of convergence is metrisable by defining Gromov--Hausdorff distance functions on the space of metric pairs that induce such convergence. In particular, this implies that the bottleneck functor $\mathcal{D}_\infty$ considered in \cite[Theorem A]{CGGGMS} is continuous and not only sequentially continuous. This is the content of section \ref{sec:GH for metric pairs}.
	
	The embedding, completeness,  
	and compactness theorems are fundamental results in the classical theory of Gromov--Hausdorff convergence \cite{burago,gromov1999metric}. Based on \cite{herron}, in section \ref{sec:main results} we prove versions of these theorems for the class of metric pairs. Moreover, we establish analogous results for a larger class, namely, the class of metric tuples. We do this in section \ref{sec:tuples}. Here we include the statements of these results in the setting of metric pairs.
	
	\begin{theorem}[Embedding theorem]\label{thm:embedding}
		Let $\{(X_i,A_i)\}_{i\in\NN}$ be a sequence of proper metric pairs. Suppose that
		\[
		\sum_{i=1}^{\infty} d_{\GH}((X_i,A_i),(X_{i+1},A_{i+1})) < \infty.
		\]
		Then there exist a non-complete locally complete metric space $Y$ and a closed subset $W\subset \overline{Y} \setminus Y=:Z$, where $\overline{Y}$ is the metric completion of $Y$, with the following properties:
		\begin{enumerate}
			\item For each $i$, the space $X_i$ naturally isometrically embeds into $Y$.
			\item The space $\overline{Y}$ is proper.
			\item $(X_i,A_i)\toGH (Z,W)$.
			\item For all $R>0$, there are $R_i>R$ such that $R_i\to R$ and $\BB_{R_i}(W)\cap X_i \xrightarrow[]{\H} \BB_R(W)\cap Z$. Moreover, if all $X_i$ are length, then $Z$ is a geodesic, and in this setting,
			\item For all $R>0$, $\BB_R(A_i)\cap X_i\xrightarrow{\H} \BB_R(W)\cap Z$.
		\end{enumerate}
	\end{theorem}
	
	\begin{theorem}[Completeness theorem]\label{thm:completeness}
		The space of all isometry classes of proper metric pairs $(\GHPair,d_\GH)$ is a complete metric space.
	\end{theorem}
	
	\begin{theorem}[Compactness theorem]\label{thm:precompactness}
		For any collection $\mathcal{X}$ of (isometry classes of) proper metric pairs that is uniformly bounded in the sense of pairs, that is, if there exists some $C>0$ such that $\diam(A)\leq C$ for any $(X,A)\in\mathcal{X}$, the following assertions are equivalent:
		\begin{enumerate}
			\item $\mathcal{X}$ is precompact with respect to $d_{\GH}$.
			\item There exists $\pi\colon (0,\infty)\to (0,\infty)$ such that for all $\eps>0$,
			\[
			P(\eps, \BB_{1/\eps}(A))\leq \pi(\eps)
			\]
			for all $(X,A)\in \mathcal{X}$, where $P$ is the packing number function (see Definition \ref{def:covering numbers and so on}).
			\item There exists $\nu\colon (0,\infty)\to (0,\infty)$ such that for all $\eps>0$,
			\[
			N(\eps, \BB_{1/\eps}(A))\leq \nu(\eps)
			\]
			for all $(X,A)\in \mathcal{X}$, where $N$ is the inner covering number function (see Definition \ref{def:covering numbers and so on}).
		\end{enumerate}
	\end{theorem}

	The extension of the theory of Gromov--Hausdorff convergence to the class of metric pairs allows for a finer study of classes of metric spaces where distinguished subsets arise naturally. In this direction, and as an application of our scheme, in section \ref{sec:applications} we obtain a formulation of the convergence of extremal sets of Alexandrov spaces under the usual Gromov--Hausdorff convergence \cite{petrunin} using the language of convergence of tuples.
	
	\begin{theorem}\label{thm:extremal}
		Let $\{X_i\}_{i\in\NN}$ and $X$ be elements in $\calA(n,k)$, the class of $n$-dimensional Alexandrov spaces with curvature bounded below by $k\in\RR$, such that $X_i\toGH X$. Let $E^m_i\supseteq  \dots \supseteq E^0_i$ be nested sequences of extremal sets in $X_i$ and $E^m\supseteq \dots \supseteq E^0$  be sets in $X$ such that $E_i^l\to E^l$ for each $l\in\{1,\dots, m\}$. Then each $E^l$ is an extremal set in $X$ and 
		\[
		(X_i,E^m_i,\dots, E^0_i) \toGH (X,E^m,\dots, E^0).
		\]
	\end{theorem}
	
	We also prove a relative version of Fukaya's theorem on the convergence of quotient spaces under equivariant convergence \cite{fukaya1}.
	
	\begin{theorem}\label{thm:equivariant-convergence}
		Let $\{X_i\}_{i\in\NN}$ be a sequence in $\calM(n,k,D)$, the class of closed $n$-dimensional Riemannian manifolds with sectional curvature bounded below by $k\in\RR$ and diameter bounded above by $D>0$, and $X$ in $\calA(n,k,D)$, the class of closed $n$-dimensional Alexandrov spaces with curvature bounded below by $k\in\RR$ and diameter bounded above by $D>0$. Let $G$ be a Lie group acting by isometries on each $X_i$ and on $X$. Let $H_0 \leq \dots \leq H_m$ be a sequence of subgroups of $G$. If $(X_i,G) \xrightarrow[]{eGH} (X,G)$ then 
		\[
		(X_i/G,X_i^{H_0}/G,\dots, X_i^{H_m}/G)\toGH (X/G,X^{H_0}/G,\dots, X^{H_m}/G).
		\]
		In the case where $\{ X_i\}_{i\in\NN}$ is a sequence in $\calN(n,k,D,v)$, the class of closed $n$-dimensional Riemannian manifolds with the absolute value of the sectional curvature bounded above by  $k>0$ and diameter bounded above by $D>0$, and $X$ is also in $\calN(n,k,D,v)$ such that $(X_i,G) \xrightarrow[]{eGH} (X,G)$ then 
		\[
		(X_i,X_i^{H_0},\dots,X_i^{H_m})\toGH (X,X^{H_0},\dots,X^{H_m}).
		\]
	\end{theorem}
	
	Finally, we get a relative version of Grove--Petersen--Wu's finiteness theorem \cite{Grove1991} for tuples induced by stratified manifolds with extremal strata.
	
	\begin{theorem}\label{thm:stratified-spaces}
		The class $\calS(n,k,D,v)$ of $n$-dimensional stratified manifolds with curvature bounded below by $k\in\RR$ in the sense of Alexandrov, diameter bounded above by $D>0$, and volume bounded below by $v>0$ is precompact in $\GHTuple_n$. Moreover, the subclass $\calS^e(n,k,D,v)$ of stratified manifolds satisfying the same conditions as before, where additionally all the strata are extremal sets, contains only finitely many relative topological types.
	\end{theorem}

	\section{Gromov--Hausdorff convergence of metric pairs}\label{sec:GH for metric pairs}
	Let us recall some basic definitions from the theory of metric spaces.
	\begin{definition}
		A map $d\colon X\times X \to [0,\infty]$ is an \emph{extended metric} on $X$ if it is symmetric and satisfies the triangle inequality, and if $d(x, y) = 0$ if and only if $x = y$. A pair $(X,d)$ where $X$ is a set and $d$ is an extended metric on $X$ is an \emph{extended metric space}. An extended metric $d$ on $X$ is a \emph{metric} if $\Im(d) \subset [0,\infty)$, in which case $(X,d)$ is a \emph{metric space}. We denote $B^{d}_{r}(x)=B_r(x):=\left\lbrace y\in X : d(x,y)<r \right\rbrace$ the \emph{open ball} of radius $r$ around $x$, whereas $\BB_r^d(x)$ denotes the corresponding \emph{closed ball}.
		
		A  metric space $(X,\delta)$ is a \emph{length space} if
		\[
		\delta(x,y)=\inf\left\{ \calL(\gamma)   : \gamma\text{ is a continuous curve from $x$ to $y$} \right\},
		\]
		where $\calL(\gamma)$ denotes the length of the curve $\gamma$.

		A metric space $(X,\delta)$ is \emph{proper} if every closed ball $\BB_{r}(p)$ is compact for every $r>0$ and $p\in X$.
		
	\end{definition}
	
	Let us also remind the definition of the Hausdorff distance between subsets of a metric space.
	\begin{definition}
		For subsets $A$ and $B$ of a metric space $(Z,\delta)$, the \emph{Hausdorff distance} of $A$ and $B$ is defined as 
		\[
		d^{\delta}_\H (A,B) := \inf\left\{\eps > 0 : A \subset B_{\eps}(B)\text{ and } B \subset B_{\eps}(A)\right\}
		\]
		where $B^{\delta}_{\eps}(B)=B_{\eps}(B) := \left\{x \in X :\text{exists }b \in B \text{ such that } \delta(x,b) < \eps\right\}$. Here, we have used the convention that the infimum of an empty set is $+\infty$.
	\end{definition}

	We now define metric pairs, the objects we will deal with in the remainder of the article.

	\begin{definition}
		A \emph{metric pair} consists of a metric space $X$ and a closed non-empty subset $A\subset X$. We will say that a metric pair $(X,A)$ is an \emph{extended metric pair} if $X$ is an extended metric space. We will denote such metric pair by $(X,A)$.
	\end{definition}

	In \cite{CGGGMS}, the authors introduced the following definition of convergence for metric pairs. It is a natural extension of the definition of Gromov--Hausdorff convergence for pointed metric spaces (cf. \cite{burago,gromov1999metric,herron}).
	\begin{definition}\label{def:Gromov-Hausdorff-CGGGMS} 
		A sequence $\{(X_i,A_i)\}_{i\in \NN}$ of metric pairs \emph{converges in the Gromov--Hausdorff topology to a metric pair} $(X,A)$ if there exist sequences
		$\{\eps_i\}_{i\in\NN}$ and $\{R_i\}_{i\in\NN}$ of positive numbers with $\eps_i\searrow 0$, $R_i \nearrow\infty$, and 
		maps $\phi_i \colon \BB_{R_i}(A_i)\to X$ satisfying the following three conditions:
		\begin{enumerate}
			\item $\left| d_{X_i}(x,y)-d_X(\phi_i(x),\phi_i(y) \right|\leq\eps_i$ for any $x,y\in \BB_{R_i}(A_i)$;
			\item $d^{d_{X}}_{\H}(\phi_i(A_i),A)\leq\eps_i$;
			\item $\BB_{R_i}(A)\subset \BB_{\eps_i}(\phi_i(\BB_{R_i}(A_i)))$.
		\end{enumerate}
		We will denote the Gromov--Hausdorff convergence of metric pairs by $(X_i,A_i)\toGH (X,A)$.
	\end{definition}

	In the setting of metric spaces, the definition of Gromov--Hausdorff convergence is usually introduced as the convergence corresponding to the Gromov--Hausdorff distance between metric spaces. Let us recall the definition of this distance.
	
	\begin{definition}
		Let $(X,d_X)$ and $(Y,d_Y)$ be two metric spaces. The \emph{Gromov--Hausdorff distance between $X$ and $Y$} is
		\[
		d_{\GH}(X,Y):=\inf\left\{ d^{\delta}_\H(X,Y):\delta \text{ is an admissible metric on }X\sqcup Y \right\},
		\]
		where a metric $\delta$ on the disjoint union $X\sqcup Y$ is \emph{admissible} if $\left.\delta\right|_{X\times X}=d_X$ and $\left.\delta\right|_{Y\times Y}=d_Y$.
	\end{definition}
	
	The previous definition allows for spaces with infinite Gromov--Hausdorff distance. However, when restricted to compact spaces, one obtains a finite distance function (cf. \cite{burago}). When dealing with non-compact spaces, though, it is customary to use the pointed Gromov--Hausdorff distance (cf. \cite{burago,herron}).
	
	One goal of this article is to prove that the convergence of metric pairs in Definition \ref{def:Gromov-Hausdorff-CGGGMS} can also be metrised. In order to do this, we consider two cases: when metric pairs are compact and when they are not.
	
	\subsection{Compact case}

	Let us first consider the case where the metric spaces are compact.
	\begin{definition}
		Let $(Z,\delta)$ be a metric space, $X,Y \subset Z$ subsets and $A \subset X$, $B \subset Y$ non-empty closed subsets. The \emph{Hausdorff distance} between $(X,A)$ and $(Y,B)$ is given by 
		\[
		d^\delta_{\H} ((X,A),(Y,B)) := d^\delta_\H (X,Y) + d^\delta_\H(A,B)
		\]
	\end{definition}
	
	\begin{definition}
		The \emph{Gromov--Hausdorff distance} between two compact metric pairs $(X,A)$ and $(Y,B)$ is defined as
		\[
		d_{\GH} ((X,A),(Y,B)) := \inf\{d^\delta_\H ((X,A),(Y,B)) : \delta\ \text{admissible on}\ X \sqcup Y\}.
		\]    
	\end{definition}
	
	One typically studies the Gromov--Hausdorff distance from a quantitative point of view through approximations. We now define the corresponding notion for metric pairs.
	\begin{definition}
		Let $X$ and $Y$ be metric spaces and $\eps > 0$. A pair of maps $f \colon X \to Y$ and $g \colon Y \to X$ (not necessarily continuous) is an \emph{$\eps$-(Gromov--Hausdorff) approximation} if for every $x,x_1,x_2 \in X$ and $y,y_1,y_2 \in Y$, 
		\begin{align*}
			|d_X (x_1,x_2) - d_Y(f(x_1),f(x_2))| &< \eps, & 
			d_X (g \circ f(x),x) < \eps,\\
			|d_Y (y_1,y_2) - d_X(g(y_1),g(y_2))| &< \eps, &
			d_Y(f \circ g(y),y) < \eps.
		\end{align*}
		The set of all such pairs is denoted by $\Appr_\eps (X,Y)$. In the case of metric pairs, one restricts to pair maps as follows: For metric pairs $(X,A)$ and $(Y,B)$, we let 
		\[
		\Appr_\eps((X,A),(Y,B)) := \{(f,g) \in \Appr_\eps(X,Y) : d_\H(f(A),B)< \eps\ \text{and}\ d_\H(g(B),A)<\eps\}.
		\]
	\end{definition}
	
	\begin{remark}
		In the literature, Gromov--Hausdorff approximations often are not defined as pairs of maps but as one map $f \colon X \to Y$ where $f$ has \emph{distortion} less than $\eps$, i.e. for all $x_1,x_2 \in X$ the map $f$ satisfies $|d_Y (f(x_1),f(x_2)) - d_X(x_1,x_2)| < \eps$, and $B_\eps(f(X)) = Y$ (compare with the maps $\phi_i$ in Definition \ref{def:Gromov-Hausdorff-CGGGMS}). Observe that $(f,g) \in \Appr_\eps(X,Y )$ already implies that $f$ has these properties (for the same $\eps$). In the following, we will see that Gromov--Hausdorff distance less than $\eps$ corresponds to the existence of $\eps$-approximations (up to a factor). The next proposition shows that (up to another factor) the definition of Gromov--Hausdorff approximations used here can be replaced by the one described in this remark.
	\end{remark}    
	
	\begin{proposition}\label{prop:key-prop}
		Let $f \colon (X,d_X) \to (Y,d_Y)$ be a map between metric spaces with distortion smaller than $\eps > 0$. Then there exists a map $g \colon f(X) \to X$ satisfying $(f,g) \in \Appr_\eps(X,f(X))$. Moreover, if $Y = B_\eps(f(X))$ and $d_\H(f(A),B)<\eps$, then there exists a map $h \colon Y \to X$ such that $(f,h) \in \Appr_{3\eps}((X,A),(Y,B))$.
	\end{proposition}
	\begin{proof}
		We define $g$ choosing some $g(y) \in f^{-1}(y)$ for $y\in f(X)$. We note that $f \circ g = \left.\Id\right|_{f(X)}$. For $y_1,y_2 \in f(X)$, 
		\[\left|d_X(g(y_1),g(y_2)) - d_Y (y_1,y_2)\right| =\left|d_X(g(y_1),g(y_2)) - d_Y (f(g(y_1)),f(g(y_2)))\right| < \eps,
		\]
		and for $x \in X$, 
		\[
		d_X(x,g \circ f(x)) = \left|d_X(x,g \circ f(x)) - d_Y(f(x),f(g \circ f(x)))\right| < \eps.
		\]
		These two inequalities are satisfied because $f$ has distortion less than $\eps$. The remaining two inequalities are satisfied trivially. 
		Thus, $(f,g) \in \Appr_{\eps}(X,f(X))$. 
		
		If $Y = B_{\eps}(f(X))$, we define 
		\[
		h(y) := 
		\begin{cases} 
			g(y) & \text{if }\ y \in f(X),\\
			g(y') & \text{if }\ y\not\in f(X)\ \text{and}\ y'\in f(X)\ \text{is such that}\ d_Y(y,y') < \eps.
		\end{cases}
		\]We have that $h\circ f =g\circ f$ and then for all $x \in X$, 
		\[
		d_X(h \circ f(x),x) < \eps.
		\]
		Now for $y \in Y$, using $f \circ g = \left.\Id\right|_{f(X)}$ or $f \circ h(y) = f \circ g(y') = y'$ for $y' \in f(X)\cap B_\eps(y)$ as in the definition of $h$, we get 
		\[
		d_Y (f \circ h(y),y) = d_Y(y',y) < \eps.
		\]
		Regarding the distortion of $h$ for every $y_1,y_2 \in Y$, 
		\begin{eqnarray*}
			\left|d_X(h(y_1),h(y_2)) - d_Y (y_1,y_2)\right| &\leq&\left|d_X(h(y_1),h(y_2)) - d_Y (f(h(y_1)),f(h(y_2)))\right|\\ \nonumber
			&       &+\left|d_Y (f(h(y_1)),f(h(y_2))) - d_Y (y_1,y_2)\right|\\ &<&\eps+d_Y(f \circ h(y_1),y_1) +d_Y(f \circ h(y_2),y_2)\\ &<& 3\eps.
		\end{eqnarray*}
		Finally we can prove that $d_\H(h(B),A)<3\eps$ as follows: if $b\in B$ then we know there is some $a\in A$ such that $d(f(a),b)<\eps$ because $d_\H(f(A),B)<\eps$, therefore we get
		\[
		d(h(b),a)<\eps+d(f\circ h(b),f(a))\leq \eps+d(f\circ h(b),b)+d(f(a),b)<3\eps 
		\] since $f$ has distortion less than $\eps$. On the other hand, if $a \in A$ then 
		\[
		d(a,h\circ f(a))<\eps
		\] as we have seen previously. Thus, we have $h(B)\subset B_{3\eps}(A)$ and $A\subset B_{\eps}(h(B))$ which implies the claim.
	\end{proof}

	\begin{proposition}\label{prop:dis-aprox}{}
		Let $(X,A)$ and $(Y,B)$ be 
		metric pairs and $\eps > 0$. Then the following assertions hold:
		\begin{enumerate}
			\item If $d_{\GH} ((X,A),(Y,B)) < \eps$, then $\Appr_{2\eps} ((X,A),(Y,B)) \neq \varnothing$.
			\item If $\Appr_\eps ((X,A),(Y,B)) \neq \varnothing$, then $d_{\GH} ((X,A),(Y,B)) \leq  4\eps$.
		\end{enumerate}
	\end{proposition}
	\begin{proof}
		To prove the first claim, we take a number $\theta$ such that $0<\theta<\eps-d_{\GH}((X,A),(Y,B))$. By the definition of infimum, we have an admissible metric $\delta$ with
		\[
		d^{\delta}_{\H}((X,A),(Y,B))<d_{\GH}((X,A),(Y,B))+\theta<\eps.
		\]
		Then $d^{\delta}_{\H}(X,Y)<\eps$ and $d^{\delta}_{\H}(A,B)<\eps$. These inequalities imply the following:
		\begin{enumerate}
			\item For every $x\in X$, there exists $y_{x}\in Y$ such that $\delta(x,y_{x})<\eps$.
			\item For every $a\in A$, there exists $b_{a}\in B$ such that $\delta(a,b_{a})<\eps$.
			\item For every $y\in Y$, there exists $x_{y}\in X$ such that $\delta(y,x_{y})<\eps$.
			\item For every $b\in B$, there exists $a_{b}\in A$ such that $\delta(b,a_{b})<\eps$.
		\end{enumerate}
		With these properties in hand, we define $f\colon X\to Y$ and $g\colon Y\to X$ by setting
		\begin{align*}
			f(x)=
			\begin{cases}
				b_{x}&\text{if }x\in A,\\
				y_{x}&\text{if } x\in X\smallsetminus A,
			\end{cases}
		\end{align*}
		\begin{align*}
			g(y)=
			\begin{cases}
				a_{y}&\text{if } y\in B, \\
				x_{y}&\text{if } y\in Y\smallsetminus B.
			\end{cases} 
		\end{align*}
		By the definition of $f$, $\delta(f(x),x)<\eps$ for every $x\in X$. Thus,
		\[
		\left| d_Y(f(x),f(x'))-d_X(x,x') \right|\leq \delta(f(x),x)+\delta(f(x'),x')<2\eps
		\]
		for every $x,x'\in X$. Analogously, 
		\[
		\left| d_X(g(y),g(y'))-d_Y(y,y') \right|< 2\eps,
		\]
		for every $y,y'\in Y$. Now,
		\begin{eqnarray*}
			d_{X}(g\circ f(x),x)  &=&\delta(g\circ f(x),x)\\
			& \leq&\delta(g\circ f(x),f(x))+\delta(f(x),x)\\
			&<& 2\eps,
		\end{eqnarray*}
		because $\delta(g(y),y)<\eps$ by definition. Also, $d_Y(f\circ g(y),y)<2\eps$.
		
		Finally, we notice that $f(A)\subset B\subset B_\eps(B)$ and $B\subset B_{2\eps}(f(A))$ because for any $b\in B$ we have $d_Y(b,f(g(b))) <2\eps$ due to the previous argument. Therefore, $d_\H(f(A),B)<2\eps$ and, in a similar way, we can prove that $d_\H(g(B),A)<2\eps$. Thus,
		\[
		(f,g)\in \Appr_{2\eps}((X,A),(Y,B)).
		\]
		
		In order to prove the second claim, we take $(f,g)\in \Appr_{\eps}((X,A),(YB))$. We define an admissible metric $\delta\colon\left( X\bigsqcup Y\right)\times\left( X\bigsqcup Y\right)\to \RR$ by setting
		\[\delta(y,x):=\delta(x,y):=
		\begin{cases}
			d_X(x,y)&\text{if }x\in X,\, y\in X,\\
			d_Y(x,y)&\text{if }x\in Y,\, y\in Y,\\
			\frac{\eps}{2}+\inf\left\{  d_X(x,x')+d_Y(f(x'),y):x'\in X \right\}&\text{if }x\in X,\, y\in Y.
		\end{cases}
		\]
		
		By definition, $ \delta $ is symmetric and positive definite. To prove the triangle inequality, first we take $x_1,x_2\in X$ and $y\in Y$. Then
		
		\begin{eqnarray*}
			\delta(x_1,x_2)+\delta(x_2,y)&=&d_X(x_1,x_2)+\frac{\eps}{2}+\inf\left\{  d_X(x_2,x')+d_Y(f(x'),y):x'\in X \right\}\\
			&=&\frac{\eps}{2}+\inf\left\{ d_X(x_1,x_2)+ d_X(x_2,x')+d_Y(f(x'),y):x'\in X \right\}\\
			&\geq&\frac{\eps}{2}+\inf\left\{ d_X(x_1,x')+d_Y(f(x'),y):x'\in X \right\}\\
			&=&\delta(x_1,y)
		\end{eqnarray*}
		and
		\begin{eqnarray*}
			\delta(x_1,y)+\delta(y,x_2)&=&\eps+ \inf\left\{ d_X(x_1,x')+d_X(x_2,x'')+d_Y(f(x'),y) +d_Y(f(x''),y):x',x''\in X \right\}\\
			&\geq& \eps+ \inf\left\{ d_X(x_1,x')+d_Y(f(x'),f(x''))+d_X(x_2,x''):x',x''\in X \right\}\\
			&\geq& \eps+ \inf\left\{ d_X(x_1,x')+(d_X(x',x'')-\eps)+d_X(x_2,x''):x',x''\in X \right\}\\
			&\geq&\inf\left\{ d_X(x_1,x_2):x',x''\in X \right\}\\
			&=&\delta(x_{1},x_2).
		\end{eqnarray*}
		For $x\in X$ and $y_1,y_2\in Y$,
		\begin{eqnarray*}
			\delta(x,y_1)+\delta(y_1,y_2)&=& \frac{\eps}{2}+\inf\left\{  d_X(x,x')+d_Y(f(x'),y_1):x'\in X \right\}+d_Y(y_1,y_2)\\
			&=& \frac{\eps}{2}+\inf\left\{  d_X(x,x')+d_Y(f(x'),y_1)+d_Y(y_1,y_2):x'\in X \right\}\\
			&\geq&\frac{\eps}{2}+\inf\left\{  d_X(x,x')+d_Y(f(x'),y_2):x'\in X \right\}\\
			&=&\delta(x,y_2)
		\end{eqnarray*}
		and
		\begin{eqnarray*}
			\delta(x,y_1)+\delta(x,y_2)&=& \eps+\inf\left\{  d_X(x,x')+d_Y(f(x'),y_1):x'\in X \right\}\\
			& & +\inf\left\{  d_X(x,x'')+d_Y(f(x''),y_2):x''\in X \right\}\\
			&=&\eps+\inf\left\{  d_X(x,x')+d_Y(f(x'),y_1)+ d_X(x,x'')+d_Y(f(x''),y_2):x',x''
			\in X \right\}\\
			&\geq&\eps+\inf\left\{  d_X(x',x'')+d_Y(f(x'),y_1) +d_Y(f(x''),y_2):x',x''\in X \right\}\\
			&\geq& \eps+\inf\left\{  \left(d_Y(f(x'),f(x''))-\eps\right)+d_Y(f(x'),y_1) +d_Y(f(x''),y_2):x',x''\in X \right\}\\
			&\geq&\inf\left\{ d_Y(y_1,y_2):x',x''\in X \right\}\\
			&=&\delta(y_1,y_2).
		\end{eqnarray*}
		
		Using the metric $\delta$, we get, for $x\in X$,
		\[
		\delta(x,f(x))=\frac{\eps}{2}+\inf\left\{ d_X(x,x')+d_Y(f(x'),f(x)):x'\in X \right\}=\frac{\eps}{2}
		\]
		using $x'=x$. For $y\in Y$, we have
		\[
		\delta(y,g(y))\leq \delta(y,f\circ g(y))+\delta(f\circ g(y),g(y))<\eps+\frac{\eps}{2}=\frac{3\,\eps} {2}
		\]
		by the previous inequality and the definition of an $\eps$-approximation. We note that these two inequalities are true when we take $x\in A$ and $y\in B$, respectively. Since $d_\H(f(A),B)<\eps$ and $d_\H(g(B),A)<\eps$, we obtain $A\subset B^{\delta}_{\eps/2}(f(A))\subset B^{\delta}_{3\eps/2}(B)$ and $B\subset B^{\delta}_{3\,\eps/2}(g(B))\subset B^\delta_{5\,\eps/2}(A)$. It is also true that $X\subset B^{\delta}_{\eps/2}(f(X))\subset B^{\delta}_{\eps/2}(Y)$ and $Y\subset B^{\delta}_{3\,\eps/2}(X)$. Putting all together, we obtain
		\[
		d_{\GH}((X,A),(Y,B))\leq d^{\delta}_{\H}((X,A),(Y,B))=d^{\delta}_{\H}(X,Y)+d^{\delta}_{\H}(A,B)\leq \frac{3\,\eps}{2}+\frac{5\,\eps}{2}=4\,\eps.
		\]
	\end{proof}
	
	\begin{definition}
		Two metric pairs $(X,A)$ and $(Y,B)$ are \emph{isometric} if there exists an isometry $f\colon X\to Y$ with $f(A)=B$.
	\end{definition}
	
	\begin{theorem}\label{teoiso}{}
		On the space of isometry classes of compact metric pairs, $d_{\GH}$ defines a metric.
	\end{theorem}
	The proof of this theorem is the same as \cite[Proposition 1.6]{jansen} using Lemma \ref{prop:dis-aprox}.
	
	We can compare the usual Gromov--Hausdorff distance with its metric pair analogue as follows.
	\begin{proposition}{}
		Let $X$ and $Y$ be compact metric spaces. Then the following assertions hold: 
		\begin{enumerate}
			\item $d_{\GH} (X,Y) \leq d_{\GH} ((X,A),(Y,B))$ for any non-empty closed sets $A\subset X$ and $B\subset Y$.
			\item For any non-empty closed subset $A\subset X$ and for any $n\in\NN$, there exists $B_n \subset Y$ such that
			\[
			d_{\GH} ((X,A),(Y,B_n)) \leq 2 d_{\GH} (X,Y)+\frac{2}{n}.
			\]
		\end{enumerate}
	\end{proposition}
	\begin{proof}
		We get both statements from the definitions. For the first one, we take closed subsets $A\subset X$ and $B\subset Y$ and we obtain
		\begin{eqnarray*}
			d_{\GH}(X,Y)&=&\inf\left\{ d^{\delta}_\H(X,Y):\delta \text{ is an admissible metric on }X\sqcup Y \right\}\\
			&\leq& \inf\left\{ d^{\delta}_\H(X,Y)+d^{\delta}_\H(A,B):\delta \text{ is an admissible metric on }X\sqcup Y \right\}\\
			&=& d_{\GH}((X,A),(Y,B)).
		\end{eqnarray*}
		
		Now, to prove the second assertion, we let $r=d_{\GH}(X,Y)$. For any $n\in\NN$, there exists an admissible metric $\delta_{n}$ on $X\sqcup Y$ satisfying
		\[
		d^{\delta_{n}}_{\H}(X,Y)<d_{\GH}(X,Y)+\frac{1}{n}=r+\frac{1}{n}.
		\]
		Thus, $X\subset\BB^{\delta_{n}}_{r+1/n}(Y)$. Now, if we fix a non-empty closed subset $A\subset X$, then for any $a\in A$ there exists $b^a_{n}\in Y$ such that 
		\[
		\delta_{n}(a,b^a_{n})\leq r+\frac{1}{n}.
		\]
		If $B_n := \{b^a_n:a\in A\}$, then 
		\[
		d_\H^{\delta_n}(A,B_n) \leq r+\frac{1}{n}.
		\]
		Thus
		\begin{eqnarray*}
			d^{\delta_{n}}_\H((X,A),(Y,B_n))&=&d^{\delta_{n}}_\H(X,Y)+d^{\delta_{n}}_\H(A,B_n)\\
			&\leq&r+\frac{1}{n}+r+\frac{1}{n}\\
			&=&2r+\frac{2}{n}.
		\end{eqnarray*}
		and 
		\begin{equation*}
			d_{\GH}((X,A),(Y,B_n))\leq 2d_{\GH}(X,Y)+\frac{2}{n}.
		\end{equation*}
	\end{proof}
	
	\begin{corollary}
		Let $X$ and $X_i$, $i \in \NN$, be compact metric spaces.
		\begin{enumerate}
			\item If $(X_i,A_i) \toGH (X,A)$ for some $A_i\subset X_i$ and $A \subset X$, then $X_i \toGH X$ as well.
			\item If $X_i \toGH X$ and $A\subset X$, then there exist $A_i\subset X_i$ such that $(X_i,A_i)\toGH (X,A)$.
		\end{enumerate}
	\end{corollary}

	We now prove that the Gromov--Hausdorff convergence of compact metric pairs can be metrised.
	
	\begin{proposition}
		If $(X,A)$ and $\{(X_i,A_i)\}_{i\in\NN}$ are compact metric pairs, $(X_i,A_i)\toGH (X,A)$ is equivalent to 
		\begin{equation}\label{eq:convergence-jansen}
			\lim_{i\to \infty} d_{\GH}((X_i,A_i),(X,A))=0.
		\end{equation}
	\end{proposition}
	
	\begin{proof}
		Let us assume that $X$ is compact and $(X_i,A_i) \toGH (X,A)$. Then we have $\eps_i \searrow 0$, $R_i \nearrow \infty$ and $f_i\colon \BB_{R_i}(A_i)\to X$ as in Definition \ref{def:Gromov-Hausdorff-CGGGMS}. Since $X$ is compact, we know that $X=\BB_{R_i}(A)$ for $i\in \NN$ sufficiently large. By the triangle inequality and the conditions (1) and (2) in Definition \ref{def:Gromov-Hausdorff-CGGGMS}, we get 
		\[
		|\diam(\BB_{R_i}(A_i))-\diam(\BB_{R_i}(A))|\leq 3\eps_i.
		\] Thus, there exists $C>0$ such that $\diam(\BB_{R_i}(A_i))< C$ for all $i\in \NN$. This condition implies that $X_i = \BB_{R_i}(A_i)$ for $i\in \NN$ sufficiently large; otherwise, we would have that $\diam(X_i)> R_i$ for arbitrarily large $i\in \NN$, which due to the fact that $f_i$ has distortion less than $\eps_i$ implies $\diam(X) > R_i - \eps_i$, and this is not possible if $X$ is compact. In particular, $f_i\colon X_i\to X$ satisfies the hypothesis of Proposition \ref{prop:key-prop} for sufficiently large $i\in \NN$, which implies that $\Appr_{3\eps_i}((X_i,A_i),(X,A))\neq \varnothing$. Thanks to Proposition \ref{prop:dis-aprox} we get $d_\GH((X_i,A_i),(X,A))\to 0$.
		
		Conversely, by \eqref{eq:convergence-jansen} and Proposition \ref{prop:dis-aprox}, we have $\Appr_{\eps}((X_i,A_i),(X,A))\neq \varnothing$ for any $\eps>0$ and sufficiently large $i\in \NN$. In particular, if we take $R_i = \diam(X)+i$, any sequence $\eps_i\searrow 0$ such that $d_\GH((X,A_i),(X,A))\leq \eps_i/2$, and $f_i\colon \BB_{R_i}(A_i)\to X$ such that there exists $g_i\colon X\to X_i$ with $(f_i,g_i)\in \Appr_{\eps_i}((X_i,A_i),(X,A))$, we get that $(X_i,A_i)\toGH (X,A)$.
	\end{proof}

	\subsection{Proper length spaces}
	
	In general, the distance function $d_\GH$ is not well-defined for non-compact metric pairs. However, we can use the distance between  metric pairs of the form $(\BB_r(A),A)$ to describe the convergence of proper length metric pairs.
	
	\begin{lemma}\label{lem:length-space-balls}
		Let $(X,\delta)$ be a proper length space, $A\subset X$ a closed subspace and $r,s>0$. Then
		\[
		B_r(B_s(A))=B_{r+s}(A).
		\]
	\end{lemma}
	\begin{proof}
		Let $q\in B_r(B_s(A))$. There exists $x\in B_s(A)$ with $\delta(x,q)<r$. Then 
		\[
		\delta(q,A)\leq \delta(q,x)+\delta(x,A)<r+s.
		\]
		Thus, $B_r(B_s(A))\subset B_{r+s}(A)$.
		
		Conversely, we take $q\in B_{r+s}(A)$. Because $B_s(A)\subset B_r(B_s(A))$, we can assume without loss of generality that $q\in B_{r+s}(A)\smallsetminus B_s(A)$. We set $l=\delta(q,A)$ and we note that $s<l<r+s$. Since $A$ is closed and $X$ is proper, we can take a shortest geodesic $\gamma$ from $A$ to $q$, i.e. $\gamma\colon [0,l]\to X$ with $\gamma(0)\in A$ and $\gamma(l)=q$. We define 
		\[
		\eps:=\frac{1}{2}\min\left\{ s,r+s-l \right\}>0
		\]
		and
		\[
		t:=s-\eps\in(0,s)\subset [0,l].
		\]
		Then $\delta(\gamma(t),A)=t<s$ and $\delta(\gamma(t),q)=l-t=l-s+\eps<l-s+r+s-l=r$. Therefore, $\gamma(t)\in B_s(A)$ and $q\in B_r(\gamma(t))$, and finally, $B_{r+s}(A)\subset B_r(B_s(A)).$
	\end{proof}
	
	\begin{lemma}\label{lem:H-balls}
		Let $(X,\delta)$ be a proper length space, $A,B\subset X$ be closed subsets, and let $r,s > 0$. Then
		\[
		d^\delta_\H (\BB_r(A),\BB_s(B))\leq d^\delta_\H(A,B) + |r - s|.
		\]
	\end{lemma}
	\begin{proof}
		We start by defining $\eps:=d^{\delta}_\H(A,B)+|r-s|\geq 0$. We have two cases.
		
		If $\eps=0$, then $d^\delta_\H(A,B)=0$ and $r=s$. Then, for any $\eps'>0$ we have 
		\[\BB_r(A)\subset \BB_{\eps'+r}(B) = \BB_{\eps'}(\BB_r(B))
		\]
		and
		\[\BB_r(B)\subset \BB_{\eps'+r}(A)=\BB_{\eps'}(\BB_r(A)),
		\]
		so $d^\delta_\H(\BB_r(A),\BB_r(B))=0$ as well.
		
		If $\eps>0$, we apply Lemma \ref{lem:length-space-balls}, and obtain
		\[
		B_r(A)\subset B_{d^{\delta}_\H(A,B)+r}(B)\subset B_{d^{\delta}_\H(A,B)+|r-s|+s}(B)\subset B_{\eps+s}(B)\subset B_{\eps}(B_s(B))
		\]
		and 
		\[
		B_s(B)\subset B_{d^{\delta}_\H(A,B)+s}(A)\subset B_{d^{\delta}_\H(A,B)+|r-s|+r}(A)\subset B_{\eps+r}(A)\subset B_{\eps}(B_r(A)).
		\]
		Therefore,
		\[
		d^{\delta}_\H(\BB_r(A),\BB_s(B))=d^{\delta}_\H(B_r(A),B_s(B))\leq \eps
		\]
		since $d_{\H}(\BB_r(A),B_r(A)=0.$
	\end{proof}
	
	\begin{corollary}\label{cor:pairGH-balls}
		Let $(X,\delta)$ be proper length spaces and $A,B\subset X$ be closed subsets. Then
		\begin{enumerate}
			\item $d_{\GH}((\BB_r(A),A),(\BB_s(A),A)) \leq |r - s|$, and
			\item $d_{\GH} ((\BB_r(A),A),(\BB_r(B),B)) \leq 2d_\H^\delta(A,B)$.
		\end{enumerate}
	\end{corollary}
	
	Observe that Lemmas \ref{lem:length-space-balls} and \ref{lem:H-balls} and Corollary \ref{cor:pairGH-balls} also hold if, instead of assuming that $X$ is proper, one asks that the subspaces $A,B\subset X$ are compact.
	
	\begin{proposition}
		If $(X,A)$ and $\{(X_i,A_i)\}_{i\in\NN}$ are proper metric spaces then
		\begin{equation}\label{eq:convergence-proper-jansen}
			\lim_{i\to \infty} d_{\GH}((\BB_R(A_i),A_i),(\BB_R(A),A))=0 \quad \text{for all } R>0.
		\end{equation}
		implies $(X_i,A_i)\toGH (X,A)$. If in addition $\{X_i\}_{i\in \NN}$ and $X$ are length spaces, then the converse also holds.
	\end{proposition}

	\begin{proof}
		If condition \eqref{eq:convergence-proper-jansen} holds, then, by \cite[Lemma 2.8]{jansen},  there exists $R_i\nearrow \infty$ such that
		\[
		\sup\{d_\GH((\BB_{R_i}(A_j),A_j),(\BB_{R_i}(A),A)):j\geq i\} \leq \frac{1}{R_i}.
		\]
		Therefore, taking $\eps_i = 2/R_i$, we have 
		\[
		d_\GH((\BB_{R_i}(A_i),A_i),(\BB_{R_i}(A),A))\leq\frac{\eps_i}{2}
		\] which, due to Proposition \ref{prop:dis-aprox}, implies that there exists some
		\[
		(f_i,g_i)\in\Appr_{\eps_i}((\BB_{R_i}(A_i),A_i),(\BB_{R_i}(A),A)).
		\] Such choice of $\eps_i$, $R_i$ and $f_i$ clearly satisfies Definition \ref{def:Gromov-Hausdorff-CGGGMS}.
		
		Let us now assume that $X$ and $\{X_i\}_{i\in\NN}$ are proper length spaces such that $(X_i,A_i)\toGH (X,A)$. Then we have $\eps_i\searrow 0$, $R_i\nearrow \infty$ and $f_i\colon \BB_{R_i}(A_i)\to X$ $\eps_i$-approximations as in Definition \ref{def:Gromov-Hausdorff-CGGGMS}. Now, if we fix $R>0$ and take $i$ sufficiently large such that $R_i>R$, we can define a metric $\delta_i$ on $\BB_R(A_i)\sqcup \BB_R(A)$ just as in the proof of Proposition \ref{prop:dis-aprox}:
		\[
		\delta_i(x,y):=\left\lbrace \begin{array}{cc}
			d_{X_i}(x,y),&x\in \BB_R(A_i),\, y\in \BB_R(A_i),\\
			d_X(x,y),&x\in \BB_R(A),\, y\in \BB_R(A),\\
			\frac{\eps}{2}+\inf\left\{  d_{X_i}(x,x')+d_X(f_i(x'),y):x'\in \BB_R(A_i) \right\},&x\in \BB_R(A_i),\, y\in \BB_R(A),\\
			\delta_i(y,x),&x\in \BB_R(A),\, y\in \BB_R(A_i).
		\end{array} \right.
		\]
		Clearly, $\delta_i$ is an admissible metric on $\BB_R(A_i)\sqcup \BB_R(A)$. 
		
		We can see that $A_i\subset \BB^{\delta_i}_{3\eps_i/2}(A)$ as follows: if $x\in A_i$ then $\delta_i(x,f_i(x))=\eps_i/2$, and we also know there is some $y\in A$ such that $\delta_i(f_i(x),y)\leq \eps_i$, so
		$\delta_i(x,y)\leq 3\eps_i/2$.
		
		On the other hand, we can also check that $A\subset \BB^{\delta_i}_{3\eps_i/2}(A_i)$: if $y\in A$ then there is some $x\in A_i$ such that $\delta_i(y,f_i(x))\leq \eps_i$. Therefore $\delta_i(y,x) \leq \eps_i/2+\delta_i(f_i(x),y)\leq 3\eps_i/2$. 
		
		Now we can see that $\BB_R(A_i)\subset \BB_{5\eps_i/2}(\BB_R(A))$: if $\delta_i(x,A_i)\leq R$ then, using the triangle inequality and the properties of $f_i$, we can verify that $\delta_i(f_i(x),A)\leq R+2\eps_i$, and since $X$ is a length space, there is some $y\in \BB_R(A)$ such that $\delta_i(f_i(x),y)\leq 2\eps_i$, so 
		\[
		\delta_i(x,y)\leq \delta_i(x,f_i(x))+\delta_i(f_i(x),y)\leq 5\eps_i/2.
		\]
		
		Let us now prove that $\BB_R(A)\subset \BB_{9\eps_i/2}(\BB_R(A_i))$.
		Therefore, for any $y \in \BB_R(A)$ there is some $x\in \BB_{R_i}(A_i)$ such that $d_X(y,f_i(x))<\eps_i$. Since $d_\H(f(A_i),A)\leq \eps_i$ and the distortion of $f_i$ is less than $\eps_i$, we get that
		\begin{align*}
			|d(x,A_i)-d(y,A)|&\leq |d(x,A_i)-d(f(x),f(A_i))|+|d(f(x),f(A_i))-d(y,f(A_i))|\\
			&\quad +|d(y,f(A_i))-d(y,A)|\\
			&\leq 3\eps_i.
		\end{align*}
		Therefore, $d(x,A_i)\leq d(y,A)+3\eps_i\leq R+3\eps_i$, which due to the fact that $X_i$ is a length space implies that there is some $x'\in \BB_R(A_i)$ such that $\delta_i(x,x')\leq 3\eps_i$. Then
		\[
		\delta_i(y,x')\leq \frac{\eps_i}{2}+d_{X_i}(x',x)+d_X(f_i(x),y) \leq \frac{9\eps_i}{2}.
		\]
		
		We conclude then that $d_\H^{\delta_i}(A_i,A)\leq 3\eps_i/2$ $d_\H^{\delta_i}(\BB_R(A_i),\BB_R(A))\leq 9\eps_i/2$. Therefore,
		\[
		d_\GH((\BB_R(A_i),A_i),(\BB_R(A),A))\leq \frac{9\eps_i}{2},
		\]
		which proves that condition \eqref{eq:convergence-proper-jansen} holds.
	\end{proof}

	\subsection{Non-compact case}
	The Gromov--Hausdorff distance between non-compact metric spaces is not well-defined in general. However, it is possible to define the Gromov--Hausdorff distance between non-compact pointed metric spaces (see \cite{gromov1999metric} and cf. \cite{burago}), which is slightly different from the corresponding definition in the compact case. This notion is thoroughly studied in \cite{herron, jansen}. We extend this definition to the setting of metric pairs.
	
	\begin{definition}
		Given $\eps>0$ and metric pairs $(X,A)$, $(Y,B)$, an admissible distance function $\delta$ on $X\sqcup Y$ is \emph{$(\eps; A, B)$-admissible} provided
		\[
		d^\delta_\H(A,B)<\eps, \quad \BB_{1/\eps}^{\delta}(A)\subset B^{\delta}_\eps(Y), \quad \BB^{\delta}_{1/\eps}(B)\subset B^{\delta}_\eps(X).
		\]
	\end{definition}
	
	\begin{definition}
		Let $(X,A)$ and $(Y,B)$ be metric pairs. The \emph{(truncated) Gromov--Hausdorff distance} between $(X,A)$ and $(Y,B)$ is 
		\[
		d_{\GH}((X,A),(Y,B)) = \min\left\{\frac{1}{2},\widetilde{d}_{\GH}((X,A),(Y,B))\right\},
		\]
		where
		\[
		\widetilde{d}_{\GH}((X,A),(Y,B)) =\inf\{\eps>0:\text{there exists a $(\eps;A,B)$-admissible distance $\delta$ on }X\sqcup Y\}
		\]
	\end{definition}
	
	\begin{remark}
		It is important to notice that both definitions of Gromov--Hausdorff distance between metric pair induce the same topology in the case of compact metric pairs.
	\end{remark}
	
	\begin{definition}
		Let $f \colon X \to Y$ be a map between metric spaces, let $\eps>0$, and $A\subset X$, and  $B \subset Y$ closed subsets. We say that $f$ is an \emph{$\eps$-rough isometry from
			$(X,A)$ to $(Y,B)$} if it satisfies $d^{d_Y}_\H(f(A),B)<\eps$, it has distortion less than $\eps$ and $Y\subset B^{d_Y}_{\eps}(f(X))$.    
	\end{definition}

	\begin{lemma}{}\label{lemmaherron1}
		Let $(X,A)$ and $(Y,B)$ be metric pairs.
		\begin{enumerate}
			\item If $d_\GH((X,A),(Y,B))<\eps<1/2$, then there exists a $2\eps$-rough isometry $f\colon\BB^{d_X}_{1/\eps}(A)\to Y$ from $(\BB^{d_X}_{1/\eps}(A),A)$ to $(\BB^{d_Y}_{1/\eps-2\eps}(B),B)$. 
			\item Conversely, let $R>\eps>0$ and suppose that there is an $\eps$-rough isometry $f\colon\BB^{d_X}_{R}(A)\to Y$ from $(\BB^{d_X}_{R}(A),A)$ to $(\BB^{d_Y}_{R-\eps}(B),B)$. Then $d_{\GH}((X,A),(Y,B))<\max\left\lbrace 3\eps,\frac{1}{R-\eps} \right\rbrace$.
		\end{enumerate}
	\end{lemma}
	\begin{proof}
		\begin{enumerate}
			\item We suppose that $ d_\GH((X,A),(Y,B))<\eps<1/2 $ and we take $\delta$ a $(\eps,A,B)$-admissible distance on $X \sqcup Y$. We define $f\colon\BB^{d_X}_{1/\eps}(A)\to Y$ by setting $f(x)\in Y$ with $\delta(x,f(x))<\eps$.
			
			First, we prove that the distortion of $f$ is less than $2\eps$. Let $x,y\in \BB^{d_X}_{1/\eps}(A)$. Then
			\begin{eqnarray*}
				\left| d_Y(f(x),f(y))-d_X(x,y) \right|&\leq& \left| \delta(f(x),x)+d_X(x,y)+\delta(y,f(y))-d_X(x,y) \right|\\
				&\leq& \delta(f(x),x)+\delta(f(y),y)\\
				&<&2\eps.
			\end{eqnarray*}
			
			Let $y\in \BB^\delta_{1/\eps-2\eps}(B)$. Since $\BB^\delta_{1/\eps}(B)\subset B^\delta_{\eps}(X)$, we take $x$ such that $\delta(x,y)<\eps$. Then
			\begin{eqnarray*}
				\delta(x,A)&\leq& \delta(x,y)+d_Y(y,B)+d_\H^\delta(A,B)\\
				&<&2\eps+d_Y(y,B)\\
				&\leq&2\eps+\frac{1}{\eps}-2\eps\\
				&=&\frac{1}{\eps}.
			\end{eqnarray*}
			Also,
			\[
			d_Y(f(x),y)\leq \delta(x,f(x))+\delta(x,y)<\eps+\eps=2\eps,
			\]
			and, therefore, 
			\[
			\BB^{d_Y}_{1/\eps-2\eps}(B)\subset B^{d_Y}_{2\eps}(f(\BB_{1/\eps}(A))).
			\]
			\item Let $R>\eps>0$ and let $f\colon\BB^{d_X}_R(A)\to Y$ be an $\eps$-rough isometry from $(\BB^{d_X}_R(A),A)$ to $(\BB^{d_Y}_{R-\eps}(B),B)$. We define
			\[
			\delta\colon X\sqcup Y\times X\sqcup Y\to \RR
			\]
			by
			\[
			\delta(y,x):=\delta(x,y):=
			\begin{cases}
				d_X(x,y)&\text{if }x\in X,\, y\in X,\\
				d_Y(x,y)&\text{if }x\in Y,\, y\in Y,\\
				\displaystyle\inf_{\substack{u\in B^{d_X}_R(A),\, v\in Y\\ d_Y(v,f(u))\leq\eps}}\left\lbrace d_X(x,u)+\frac{3\eps}{2}+d_Y(y,v)\right\rbrace&\text{if }x\in X,\, y\in Y.
			\end{cases}
			\]
			We will show that $\delta$ is a $(t;A,B)$-admissible distance on $X\sqcup Y$, where $t=\max\left\lbrace 3\eps,\frac{1}{R-\eps}  \right\rbrace$. Note that for $x\in \BB^{d_X}_R(A)$ and $y\in Y$, we have $\delta(x,y)\leq 3\eps/2+d_Y(f(x),y)$.
			
			It is clear that $\delta$ is symmetric and positive definite. The triangle inequality is valid where the three points lie in $X$ and or in $Y$. Now, we have several cases to check. The first is where there is one point in $X$. Suppose that $x\in X$ and $y,z\in Y$. Let $u\in \BB^{d_X}_R(A)$ and $v\in Y$ with $d_Y(v,f(u))\leq\eps.$ Then, by definition,
			\begin{eqnarray*}
				\delta(x,z)&\leq&d_X(x,u)+\frac{3\eps}{2}+d_Y(v,z)\\
				&\leq&d_X(x,u)+\frac{3\eps}{2}+d_Y(v,y)+d_Y(y,z)\\
				&=& d_X(x,u)+\frac{3\eps}{2}+d_Y(v,y)+\delta(y,z).
			\end{eqnarray*}
			The preceding inequality implies, after taking the infimum over $u \in \BB^{d_X}_R(A)$ and $v\in Y$ such that $d_Y(v,f(u))<\eps$, the triangle inequality $\delta(x,z)\leq \delta(x,y)+\delta(y,z)$.
			
			Now suppose that $z\in X$ and $x,y\in Y$. Let $u\in \BB^{d_X}_R(A)$ and $v\in Y$ with $d_Y(v,f(u))\leq\eps$. Then
			\begin{eqnarray*}
				\delta(z,x)&\leq&d_X(z,u)+\frac{3\eps}{2}+d_Y(v,x)\\
				&\leq&d_X(z,u)+\frac{3\eps}{2}+d_Y(v,y)+d_Y(y,x)\\
				&=&d_X(z,u)+\frac{3\eps}{2}+d_Y(v,y)+\delta(y,x).
			\end{eqnarray*}
			Taking the infimum over $u \in \BB^{d_X}_R(A)$ and $v\in Y$ such that $d_Y(v,f(u))<\eps$, we have $\delta(z,x)\leq \delta(z,y)+\delta(y,x)$.
			
			Suppose that $y\in X$ and $x,z\in Y$. Let $u,p\in \BB^{d_X}_R(A)$ and $v,q\in Y$ such that $d_Y(v,f(u))\leq\eps$ and $d_Y(q,f(p))\leq \eps$. Then
			\begin{eqnarray*}
				\delta(x,z)&\leq& d_Y(x,v)+d_Y(v,f(u))+d_Y(f(u),f(p))+d_Y(f(p),q)+d_y(q,z)\\
				&\leq&d_Y(x,v)+d_Y(f(u),f(p))+2\eps+d_Y(q,z)\\
				&\leq&d_Y(x,v)+d_X(u,p)+\left|d_Y(f(u),f(p))-d_X(u,p)\right|+2\eps+d_Y(q,z)\\
				&\leq&d_Y(x,v)+d_X(u,p)+3\eps+d_Y(q,z)\\
				&\leq&\left( d_X(y,u)+\frac{3\eps}{2}+d_Y(x,v) \right)+\left( d_X(y,p)+\frac{3\eps}{2}+d_Y(q,z). \right)
			\end{eqnarray*}
			Taking the infimum over $u \in \BB^{d_X}_R(A)$ and $v\in Y$ such that $d_Y(v,f(u))<\eps$, we have $\delta(x,z)\leq \delta(y,x)+\delta(y,z)$.
		\end{enumerate}
		
		The case where we have two points in $X$ and therefore one point in $Y$ is analogous.
		
		We see now that it is admissible. Let $u\in A$ and $v\in B$, then $d^{\delta}_\H(A,B)<5\eps/2<t$. If $x\in \BB^\delta_{1/t}(A)\subset \BB^\delta_R(A)$, then $\delta(x,f(x))\leq3\eps/2<t$. Thus,
		\[
		\BB^{\delta}_\frac{1}{t}(A)\subset B^\delta_t(Y).
		\]
		
		Let $y\in \BB^{d_Y}_{1/t}(B)\subset \BB^{d_Y}_{R-\eps}(B)$. Since $\BB^{d_Y}_{R-\eps}(B)\subset B^{d_Y}_\eps(f(\BB^{d_X}(A)))$, there exists $x\in \BB^{d_X}_R(A)$ such that $d_Y(f(x),y)<\eps$. Taking $x=u$ and $y=v$, we get $\delta(x,y)<3\eps/2<t$.
		
		Thus, $d_\GH((X,A),(Y,B))<t$.
	\end{proof}
	
	Since we have a distance function for metric pairs, we can talk about convergence of sequences. We have the following characterisation.

	\begin{proposition}\label{prop:Herron-equivalences}
		Let $\{(X_i, A_i)\}_{i\in\NN}$  be a sequence of metric pairs. The
		following are equivalent:
		\begin{enumerate}
			\item $d_\GH((X_i,A_i), (X,A))\to 0$.
			\item For all $R>0$, there exist $R_i>R$, $\eps_i>0$ and maps $f\colon\BB^{d_{X_i}}_{R_{i}}(A_i)\to X$ such that $R_i\to R$, $\eps_i\to 0$, and $f_i$ are $\eps_i$-rough isometries from $\left(\BB^{d_{X_i}}_{R_i}(A_i),A_i\right)$ to $\left(\BB_R^{d_X}(A),A\right)$.
			\item For all $R>\eps>0$ there is an $I\in\NN$ such that for all $i\geq I$ there are maps $f_i\colon\BB^{d_{X_i}}_R(A_i)\to X$ that are $\eps$-rough isometries from $\left(\BB^{d_{X_i}}_R(A_i),A_i  \right)$ to $\left( \BB^{d_X}_{R-\eps}(A),A \right)$.
		\end{enumerate}
	\end{proposition}
	\begin{proof}
		From Lemma \ref{lemmaherron1}, part (2), we have that assertion (3) implies assertion (1).
		
		We suppose that assertion (1) holds. Let $R>0$. We take $I\in \NN$ such that
		\[
		\eps_{i}:=2\,d_\GH((X_i,A_i),(X,A))<\min\left\lbrace \frac{1}{2},\frac{1}{R+1} \right\rbrace
		\]
		for every $i\geq I$. For $1\leq i<I$, we define $R_i:=R+1$ and $t_i:=4R$. For $i\geq I$, we define $R_i:=R+t_i$, where $t_i:=2\eps_i$. Then $R_i\to R$ and $t_i\to 0$. Also, for $i\geq I$,
		\[
		\frac{1}{\eps_i}\geq R+1\geq R+2\eps_i=R+t_i=R_i
		\]
		and there are $(\eps_i;A_i,A)$-admissible distances $\delta_i$ on $X_i\sqcup Z$.
		
		For $1\leq i< I$, we define the constant maps $f_i\colon \BB^{d_X}_{R_i}(A_i)\to X$ by $f_i(x)=a\in A$. For $i\geq I$, let $f_i(x)\in A$ be any point with $\delta_i(x,f_i(x))<\eps_i$ for $x\in A_i$ and $f_i(x)\in X\smallsetminus A$ be any point with $\delta_i(x,f_i(x))<t_i$ with $x\in \BB^{d_{X_i}}_{R_i}(A_i)\smallsetminus A$. These points always exist because 
		\[
		\BB^{\delta_i}_{R_i}(A_i)\subset \BB^{\delta_i}_{1/\eps_i}(A_i)\subset B^{\delta_i}_{\eps_i}(X)
		\] 
		and $d_\H^{\delta_i}(A_i,A)<\eps_i$.
		
		The maps $f_i$ are clearly $t_i$-rough isometries for $0\leq i <I$. We assume that $i\geq I$. Since $\delta_i(x,f_i(x))<\eps_i$ for all $x\in \BB^{d_{X_i}}_{R_i}(A_i)$, we have
		\[
		\left| d_X(f_i(x),f_i(y))-d_{X_i}(x,y) \right|\leq \delta_i(f_i(x),x)+\delta_i(f_i(y),y)\leq 2\eps_i=t_i.
		\]
		Since $\BB^{\delta_i}_R(A)\subset \BB^{\delta_i}_{1/\eps_i}(A)\subset B^{\delta_i}_{\eps_i}(X_i)$, if we take $y\in \BB^{\delta_i}_R(A)$, there exists a point $x\in X_i$ such that $\delta_i(x,y)<\eps_i$. Then 
		\[
		d_{X_i}(x,A_i)\leq \delta_i(x,y)+d_X(y,A)+d^{\delta_i}_\H(A,A_i)<R+2\eps_i=R_i.
		\]
		Therefore, $x\in \BB^{d_{X_i}}_{R_i}(A_i)$ and 
		\[
		d_X(f_i(x),y)\leq \delta_i(x,f_i(x))+\delta_i(x,y)<2\eps_i=2t_i.
		\]
		
		Finally, we suppose that assertion (2) holds and let $R>\eps>0$. We choose $R_i>R$, $t_i$ and $f_i$ as in assertion (2). Therefore, we have maps $f\colon\BB^{d_{X_i}}_{R_i}(A_i)\to X$ whose distortion is less than $t_i$ and $d_\H^{d_X}(f_i(A_i),A)<t_i$. We choose $I\in\NN$ such that, for all $i\geq I$, $t_i<\eps/3$.
		
		We take $i\geq I$. We have to see that $f_i$ is a $\eps$-rough isometry from $(\BB^{d_{X_i}}_R(A_i),A_i)$ to $(\BB^{d_X}_{R-\eps}(A),A)$ and it is left to prove that $\BB^{d_X}_{R-\eps}(A)\subset B^{d_X}_\eps(f(\BB^{d_{X_i}}_R(A_i)))$. Let $y\in B^{d_X}_{R-\eps}(A)$. Then 
		\[
		y\in B^{d_X}_{R}(A)\subset B^{d_X}_{t_i}(f_i(\BB^{d_{X_i}}_{R_i}(A_i)).
		\]
		Hence, there exists $x\in \BB^{d_{X_i}}_{R_i}(A_i)$ with $d_X(f_i(x),y)<t_i$. Thus,
		\begin{eqnarray*}
			d_{X_i}(x,A_i)&<&d_X(f_i(x),f_i(A))+t_i\\
			&\leq& d_X(f_i(x),y)+d_X(y,A)+d_\H^{d_X}(A,f_i(A_i))+t_i\\
			&<&3t_i+(R-\eps)\\
			&<& R.
		\end{eqnarray*}
		Therefore, $y\in B^{d_X}_\eps(f(B^{d_{X_i}}_R(A_i)))$.
	\end{proof}
	
	The following lemma provides another useful method for estimating the truncated Gromov--Hausdorff distance between metric pairs.
	
	\begin{lemma}[cf. Lemma 3.3 in \cite{herron}]\label{lem:GH-with-nets}
		Let $(X,A)$, $(Y,B)$ be metric pairs and $\eps>0$. Suppose there are $\{x_1,\dots,x_n\}\subset X$ and $\{y_1,\dots,y_n\}\subset Y$ such that 
		\begin{align*}
			B_{\frac{1}{2\eps}}(A)\subset \bigcup_{i=1}^{n} B_\eps(x_i),&\\
			B_{\frac{1}{2\eps}}(B)\subset \bigcup_{i=1}^{n} B_\eps(y_i),&\\
			A\subset \bigcup_{i=1}^{k} B_\eps(x_i),&\\
			A\cap B_\eps(x_i) \neq \varnothing\ \forall\ 1\leq i\leq k,&\\
			B\subset \bigcup_{i=1}^{k} B_\eps(y_i),&\\ 
			B\cap B_\eps(y_i) \neq \varnothing\ \forall\ 1\leq i\leq k, \text{ and }&\\
			\text{for all }\ i,j,\ |d_X(x_i,x_j)-d_Y(y_i,y_j)|\leq \eps.&
		\end{align*}
		Then $d_\GH((X,A),(Y,B))\leq 3\eps$.
	\end{lemma}
	
	\begin{proof}
		We define an admissible metric on $X\sqcup Y$ by setting
		\[
		\delta(x,y) := \delta(y,x) = \min_{1\leq i\leq n}\left\{d_X(x,x_i)+d_Y(y,y_i)\right\} + \eps.
		\]
		This is an actual metric and the proof is the same as in \cite[Lemma 3.3]{herron}. Moreover, if $x\in A$, then there is some $i\in\{1,\dots,k\}$ such that $x\in B_\eps(x_i)$. Then, for any $y\in B\cap B_\eps(y_i)$, we have 
		\[
		\delta(x,y) \leq d_X(x,x_i)+d_Y(y,y_i)+\eps <3\eps,
		\] which implies that $A\subset B^\delta_{3\eps}(B)$. Analogously, we have $B\subset B^\delta_{3\eps}(A)$, thus $d^\delta_H(A,B)<3\eps$. Finally, $\BB_{\frac{1}{3\eps}}(A)\subset B^\delta_{3\eps}(Y)$ and $\BB_{\frac{1}{3\eps}}(B)\subset B^\delta_{3\eps}(X)$ can be easily verified by an analogous argument to the one in \cite[Lemma 3.3]{herron}.
	\end{proof}
	
	We introduce the following definitions to understand several results from now on.
	
	\begin{definition}\label{def:covering numbers and so on}
		Let $X$ be a metric space. A subset $S\subset X$ is $\eps$-\emph{separated} if its cardinality is greater than $1$ and for all distinct $x,y\in S$, we have $d_X(x,y)\geq\eps$. For $A\subset X$ and $r>0$ we define \emph{outer and inner covering} numbers via
		\begin{align*}
			&M(r,A):=\min\left\{m\in\NN:\text{there exist }x_1,\dots,x_m\in X\text{ such that }A\subset B_r(x_1)\cup\dots\cup B_r(x_m)\right\},\\
			&N(r,A):=\min\left\{n\in\NN:\text{there exist }a_1,\dots,a_n\in A\text{ such that }A\subset B_r(a_1)\cup\dots\cup B_r(a_n)\right\},
		\end{align*}
		and \emph{packing} and \emph{separation} numbers via
		\begin{align*}
			&P(r,A):=\max\left\{p\in\NN:\text{there exist }a_1,\dots,a_p\in A\text{ such that } B_r(a_1),\dots, B_r(a_m)\text{ are disjoint}\right\},\\
			&S(r,A):=\max\left\{s\in\NN:\text{there exist }\left\{a_1,\dots,a_s\right\}\subset A\text{ such that 
				it is $r$-separated}\right\}.
		\end{align*}
	\end{definition}

	The proof of the following lemma is analogous to the proof of \cite[Lemma 3.9]{herron}.
	\begin{lemma}\label{lem:herron-3.9}
		Let $(X,A)$ and $(Y,B)$ be metric pairs with $d_\GH((X,A),(Y,B))<\eps<1/2$. Then for any $(\eps;A,B)$-admissible metric $\delta$ on $X\sqcup Y$ and all $R>0$ and $r>0$:
		\begin{align*}
			R\leq 1/\eps &\Rightarrow M(r+2\eps, \BB_R(B)\cap Y)\leq N(r,\BB_R(A)\cap X)\ \text{and}\\
			R+r\leq 1/\eps &\Rightarrow P(r+2\eps, \BB_{R-2\eps}(B)\cap Y)\leq P(r,\BB_R(A)\cap X)
		\end{align*}
	\end{lemma}
	
	We also get an analogous version of \cite[Corollary 3.10]{herron}. 
	
	\begin{corollary}\label{cor:properness}
		Let $d_{\GH}((X_i,A_i),(X,A))\to 0$. If each $X_i$ is a proper space and $X$ is complete, then $X$ is proper too.    
	\end{corollary}
	
	The proof of the previous corollary is the same as in \cite[Corollary 3.10]{herron} after fixing some $a\in A$ and observing that whenever $d_\GH((X_i,A_i),(X,A))<\eps<1/2$ then we can find $\delta_i$ a $(\eps;A_i,A)$-admissible metric on $X_i\sqcup X$ and $a_i\in A_i$ such that $\delta_i(a_i,a)<\eps$. 
	
	\begin{corollary}\label{cor:properness2}
		Let $X$ be a proper metric space and $Y$ be a complete metric space such that $d_\GH((X,A),(Y,B))=0$. Then Y is proper.
	\end{corollary}
	
	\begin{proposition}\label{prop:isom}
		Let $(X,A)$ and $(Y,B)$ be metric pairs. Suppose that one space is proper and the other is complete. Then 
		\[
		d_\GH((X,A),(Y,B))=0
		\]
		if and only if $(X,A)$ and $(Y,B)$ are isometric.
	\end{proposition}
	
	The proof of the preceding proposition is the same as that of \cite[Proposition 3.12]{herron}. We only notice that the balls $\BB_r(A)$ are separable since $\BB_r(A)$ is proper and is the countable union of compact balls $\BB_s(p)$. This fact allows us to construct the isometry between $(X,A)$ and $(Y,B)$ along the lines of the construction in \cite{herron}.

	\begin{corollary}
		Let $\GHPair$ denote the collection of all isometry classes of proper metric pairs $(X,A)$. Then $(\GHPair,d_{\GH})$ is a metric space.
	\end{corollary}
	\begin{proof}
		Clearly, $d_\GH$ is symmetric and non-negative, and satisfies the triangle inequality. From Proposition \ref{prop:isom}, $d_\GH$ is positive definite. Therefore, $(\GHPair,d_\GH)$ is a metric space.
	\end{proof}

	\begin{proposition}
		Let $\{(X_i,A_i)\}_{i\in\NN}$ and $(X,A)$ proper metric pairs. Then $(X_i,A_i)\toGH (X,A)$ is equivalent to \[\lim_{i\to \infty}d_{\GH}((X_i,A_i),(X,A))=0.\]
	\end{proposition}
	\begin{proof}
		Let us assume that $(X_i,A_i)\toGH (X,A)$, that is, we have $\eps_i\searrow 0$, $R_i\nearrow \infty$ and maps $f_i\colon \BB_{R_i}(A_i)\to X$ as in Definition \ref{def:Gromov-Hausdorff-CGGGMS}. If $R>\eps>0$, take $i\in \NN$ sufficiently large such that $R_i>R>\eps>3\eps_i$. It is clear that the restriction of $f_i$ gives an $\eps_i$-rough isometry from $(\BB_{R_i}(A_i),A_i)$ to $(\BB_{R_i}(A),A)$. Now, this implies that $f_i$ restricted to $\BB_R(A_i)$ still has distortion less than $\eps_i$ and $d_\H(f_i(A_i),A)\leq \eps_i$. Moreover, for any $y\in \BB_{R-\eps}(A)\subset \BB_{R_i}(A)$, we know there is some $x\in \BB_{R_i}(A_i)$ such that $d_X(y,f_i(x))<\eps_i$, and by the triangle inequality and the fact that $f_i$ has distortion less than $\eps_i$ and $d^{d_X}_\H(f_i(A_i),A)\leq \eps_i$, we have
		\[
		|d_{X_i}(x,A_i)-d_X(f_i(x),A)|\leq 2\eps_i,
		\]
		which in turn implies 
		\[
		d_{X_i}(x,A_i) \leq d_X(f_i(x),A)+2\eps_i \leq d_X(f_i(x),y)+d_X(y,A)+2\eps_i \leq 3\eps_i+R-\eps<R.
		\]
		Thus, $\BB_{R-\eps}(A)\subset B_\eps(f_i(\BB_{R}(A)))$. This means that $f_i$ induces an $\eps$-rough isometry from $(\BB_{R}(A_i),A_i)$ to $(\BB_{R-\eps}(A),A)$ for sufficiently large $i\in \NN$. Using Proposition \ref{prop:Herron-equivalences}, we conclude that \[d_\GH((X_i,A_i),(X,A))\to 0.\]
		
		Conversely, if we assume that 
		\[d_\GH((X_i,A_i),(X,A))\to 0\] and we consider sequences $\eps_i\searrow 0$ and $R_i \nearrow\infty$ with $R_i>\eps_i$, then we have $\eps_i$-rough isometries $f_i$ from $(\BB^{d_{X_i}}_{R_i}(A_i),A_i)$ to $(\BB^{d_X}_{R_i-\eps_i}(A),A)$, by assertion (3) of Proposition \ref{prop:Herron-equivalences}. This is exactly Definition \ref{def:Gromov-Hausdorff-CGGGMS}.
	\end{proof}

	\section{Main results}
	\label{sec:main results}
	This section is devoted to the embedding, completeness and compactness theorems for the Gromov--Hausdorff distance of metric pairs. These results are the counterparts of the main results in \cite{herron}. The proofs are natural generalisations of the arguments given in \cite{herron} but we include most of the details for the sake of completeness.

	\begin{proof}[Proof of Theorem \ref{thm:embedding}]
		We can construct the space $Y$ and prove it is a non-complete and locally complete metric space satisfying item 1 just as in the proof of the Embedding Theorem in \cite{herron}.
		Moreover, we define $\eps_n$, $R_n$ and $\delta_n$ in the same way as in the proof of the Embedding Theorem in \cite{herron}. Namely, we choose $\eps_n>d_{\GH}((X_n,A_n),(X_{n+1},A_{n+1}))$ such that $\sum_{n=1}^{\infty} \eps_n < \infty$. We also set $R_n = 1/\eps_n$ and choose $\delta_n$ a $(\eps_n;A_n,A_{n+1})$-admissible metric on $X_n\sqcup X_{n+1}$.
		
		It is clear that any sequence $\{a_i\}_{i\in \NN}\subset Y$ such that $a_i\in A_i$ and $d(a_i,a_{i+1})< \eps_i$ for sufficiently large $i\in\NN$ is a Cauchy sequence, therefore it converges to some $a\in \overline{Y}$. We can then define
		\[
		W = \left\{\lim_{i\to \infty}a_i\in \overline{Y} : \{a_i\}_{i\in \NN}\subset Y,\ a_i\in A_i\ \text{and}\ d(a_i,a_{i+1})< \eps_i\right\}.
		\]
		This set is non-empty, since each $A_i$ is non-empty and we can construct at least one limit of a sequence as in the definition of $W$. It is also a closed subset of $\overline{Y}$ by a standard diagonal argument.
		
		We will now prove that $(X_i,A_i)\toGH (Z,W)$. The argument is very similar to the one used in \cite{herron}. We give the details for the convenience of the reader.
		
		Fix $\eps \in (0,1/2)$ and set $R=1/\eps$. Take $N\in \NN$ such that $\sum_{i=n}^\infty \eps_i<\eps/4$ for $n\geq N$. 
		
		\begin{claim} 
			$d^d_\H(A_n,W)<\eps$ for $n\geq N$. 
		\end{claim}
		
		If $a\in W$ then choose any sequence $\{a_i\}_{i\in \NN} \subset Y$ such that $a_i\in A_i$, $d(a_i,a_{i+1})<\eps_i$ and $\lim_{i\to \infty} a_i = a$. In particular, if $n\geq N$ then
		\[
		d(a_n,a) = \lim_{i\to \infty} d(a_n,a_i) \leq \lim_{i\to \infty} \sum_{k=n}^{i-1}d(a_k,a_{k+1})\leq \sum_{k=n}^\infty \eps_k <\frac{\eps}{4}<\eps.
		\] Therefore $W\subset B^d_\eps(A_n)$.
		
		On the other hand, if $a_n\in A_n$ then we can inductively construct a sequence $\{a_i\}_{i=n}^{\infty}\subset Y$ such that $a_i\in A_i$ and $d(a_i,a_{i+1}) < \eps_i$ for all $i\geq n$, therefore this sequence is convergent in $\overline{Y}$, with some limit $a\in W$ and clearly $d(a_n,a) <\eps$. Thus $A_n\subset B^d_\eps(W)$, which proves the claim.
		
		\begin{claim} 
			$\BB^d_R(A_n)\cap X_n\subset B^d_\eps(Z)$ for $n\geq N$.
		\end{claim}
		
		The following is a simple consequence of the definition of $\eps_n$:
		\begin{align*}
			\BB^{\delta_n}_{R}(A_n)\cap X_n \subset B^{\delta_n}_{\eps_n}(\BB^{\delta_n}_{R+2\eps_n}(A_{n+1})\cap X_{n+1}),\\
			\BB^{\delta_n}_{R}(A_{n+1})\cap X_{n+1} \subset B^{\delta_n}_{\eps_n}(\BB^{\delta_n}_{R+2\eps_n}(A_{n})\cap X_{n})
		\end{align*}
		for any $R\in (0,R_n]$.
		
		Given any $x_n\in \BB^d_R(A_n)\cap X_n$, we can then construct a sequence $\{x_i\}_{i=n}^\infty$ with $x_i\in X_i$ and $d(x_i,x_{i+1})<\eps_i$ just as in \cite{herron}. Such a sequence is Cauchy and converges to some $x\in Z$ with $d(x_n,x)<\eps$. This implies the claim.
		
		\begin{claim} 
			$\BB^d_R(W)\cap Z\subset B^d_\eps(X_n)$ for $n\geq N$.
		\end{claim}
		
		By an analogous argument to the one in \cite[Section 4.1.3]{herron}, we can prove the following Engulfing Conditions: for any $T>0$ and $N\in\NN$ such that
		\[
		T+2\sum_{k=N}^\infty \eps_k < R_n
		\] for any $n\geq N$, we have
		\[
		\BB^d_T(A_m)\cap X_m \subset \begin{cases}
			B^d_{\sum_{k=m}^{n-1}\eps_k}(X_n) & \text{if } n>m\geq N,\\
			B^d_{\sum_{k=n}^{m-1}\eps_k}(X_n) & \text{if } m>n\geq N.
		\end{cases}
		\]
		In particular, if $T=R+2\eps$ and $N\in \NN$ is such that $\sum_{k=N}^\infty \eps_k<\eps/4$ then
		\[
		T+ 2\sum_{k=N}^\infty \eps_k < \frac{1}{\eps} + 2\eps + \frac{\eps}{2} < \frac{7}{2\eps} < R_n
		\] for any $n\geq N$. Therefore we have 
		\[
		\BB^d_{R+2\eps}(A_m)\cap X_m \subset B^d_{\eps/2}(X_n)
		\] for any $m,n\geq N$. In particular, since $\BB^d_{R+\eps}(A_N) \subset \BB^d_{R+2\eps}(A_m)$ for $m\geq N$, which can be easily verified by the definition of $N$, we get that
		\[
		\BB^d_{R+\eps}(A_N)\cap X_m \subset B^d_{\eps/2}(X_n)
		\] for $m,n\geq N$, therefore
		\[
		\BB^d_{R+\eps}(A_N)\cap \bigsqcup_{m=N}^\infty X_m \subset B^d_{\eps/2}(X_n)
		\] for $n\geq N$. This implies
		\[
		B^d_{R+\eps}(A_N)\cap Z \subset B^d_{R+\eps}(A_N)\cap \overline{\left(\bigsqcup_{m=N}^\infty X_m\right)} \subset B^d_{\eps}(X_n).
		\]
		
		Now if we fix $z\in \BB^d_R(W)\cap Z$ then there is some $w\in W$ such that $d(z,w)\leq R$, and since $w\in W$ then $w=\lim_{n\to \infty} a_n$ for some sequence $\{a_n\}_{n\in \NN}$ with $a_n\in A_n$ and $d(a_n,a_{n+1})<\eps_n$. In particular,
		\[
		d(z,a_N)\leq d(z,w) + d(w,a_N) \leq R + \lim_{m\to \infty} \sum_{k=N}^{m-1}d(a_k,a_{k+1}) \leq R+ \sum_{k=N}^\infty \eps_k < R+ \eps.
		\] Therefore $\BB_R^d(W)\cap Z\subset B^d_{R+\eps}(A_N)\cap Z\subset B^d_\eps(X_n)$ and the claim follows.
		
		Combining claims 1, 2 and 3 we can conclude that $(X_i,A_i) \toGH (Z,W)$. We prove the properness of $\overline{Y}$ by applying the same argument as in the proof of the Embedding Theorem in \cite{herron} after fixing some $w\in W$ and some sequence $\{a_i\}_{i\in \NN}$ such that $a_i\in A_i$, $d(a_i,a_{i+1})<\eps_i$ and $w=\lim_{i\to\infty} a_n$, and observing that $(X_i,a_i)\toGH (Z,w)$ in the sense of \cite{herron}.
		
		For part 4, let $R>0$ and, for each $i\in\NN$, set $R_i := R+2h_i$ where
		\[
		h_i := 2d_\GH((X_i,A_i),(Z,W)).
		\]
		Then $R_i>R$ and $R_i\to R$. Moreover, for $i$ sufficiently large, we have $h_i<1/R$ and
		\[
		\BB_R(W)\cap Z \subset B_{h_i} (B_{R+2h_i}(A_i)\cap X_i).
		\]
		In particular, for each point $w\in \BB_R(W)\cap Z$ there are points $x_i\in B_{R_i}(A_i)\cap X_i\subset \BB_{R_i}(W)\cap X_i$ with $d(x_i,w)<h_i$.
		
		On the other hand, if $\{x_i\}_{i\in\NN}$ is any sequence such that $x_i\in\BB_{R_i}(W)\cap X_i$ for all $i\in\NN$ and $\{x_{i_k}\}_{k\in \NN}$ is a subsequence that converges to $x\in Z$, we can prove that $d(x,W) \leq R$ as follows: if $\eps>0$, pick $N\in \NN$ such that whenever $k\geq N$ then
		\[
		R_{i_k}<R+\eps/3, \quad d^d_\H(A_{i_k},W)<\eps/3, \quad \text{and}\quad d(x_{i_k},x)<\eps/3,
		\] which implies that there exist $a_{i_k}\in A_{i_k}$ and $w_{i_k}\in W$ such that
		\[
		d(x,W)\leq d(x,w_{i_k})\leq d(x,x_{i_k})+d(x_{i_k},a_{i_k})+d(a_{i_k},w_{i_k})<R+\eps.
		\] Taking $\eps \to 0$, we get the claim.
		
		By the same argument as in \cite{herron}, we conclude that $\BB_{R_i}(W)\cap X_i$ Kuratowski converges to $\BB_R(W)\cap Z$, which implies that it also Hausdorff convergence since the spaces are proper.
		
		Finally, for part 5, we get that $Z$ is geodesic whenever the spaces $X_i$ are length spaces by the same argument as in \cite{herron}, and regarding the convergence of the closed sets $\BB_R(A_i)$ in $\overline{Y}$, we observe that they eventually lie in $\BB_{2R}(A)$, and since $\overline{Y}$ is proper, it is enough to prove that $\BB_R(A_i)$ Kuratowski converge to $\BB_R(A)$. The details are analogous to those in \cite{herron}.
	\end{proof}

	\begin{proof}[Proof of Theorem \ref{thm:completeness}]    
		Let $\{(X_i,A_i)\}_{i\in\NN}$ be a Cauchy sequence in $\GHPair$ an take a subsequence $\{(X_{i_k},A_{i_k})\}_{k\in\NN}$ such that
		\[
		\sum_{k=1}^{\infty} d_{\GH}((X_{i_k},A_{i_k}),(X_{i_{k+1}},A_{i_{k+1}})) < \infty.
		\] Then Theorem \ref{thm:embedding} implies that $(X_{i_k},A_{i_k})\toGH (Z,W)$ for some proper metric pair $(Z,W)$. Thus, the isometry class of $(Z,W)$ belongs to $\GHPair$ and $(X_i,A_i)\toGH (Z,W)$.
	\end{proof}

	\begin{proof}[Proof of Theorem \ref{thm:precompactness}]
		As in the preceding theorems in this section, the proof is analogous to the one in the pointed case (see \cite{herron}). We give the details for convenience of the reader. 
		
		The implication $(2)\Rightarrow (3)$ follows directly from the fact that 
		\[
		N(\eps, \BB_{1/\eps}(A))\leq P(\eps/2, \BB_{1/\eps}(A))\leq P(\eps/2, \BB_{2/\eps}(A))\leq \pi(\eps/2),
		\]
		so we can define $\nu(\eps) = \pi(\eps/2)$.
		
		On the other hand, if we assume $(1)$ and $\mathcal{X}$ is precompact with respect to $d_\GH$ and uniformly bounded in the sense of pairs, then in particular $\mathcal{X}$ is totally bounded. Therefore for a fixed $\eps\in(0,1)$ there is a minimal $N\in \NN$ (which only depends on $\eps$) such that there exist $(X_1,A_1),\dots,(X_N,A_N)\in\mathcal{X}$ that make up a $(\eps/5)$-net for $\mathcal{X}$. We then define
		\[
		\pi(\eps) = \max_{1\leq n\leq N}\{P(\eps/2,\BB_{2/\eps}(A_n)\cap X_n)\},
		\] which is finite since each $A_n$ is compact. We can verify that $\pi(\eps)$ satisfies item $(2)$ by applying Lemma \ref{lem:herron-3.9} and an analogous argument to the one used in \cite{herron}.
		
		Finally, we prove that $(3)\Rightarrow (1)$. Let us assume there is some function $\nu$ as in $(3)$ and fix $\eps>0$ and a sequence $\{(X_n,A_n)\}_{n \in\NN} \subset \mathcal{X}$. We will prove that there is a subsequence $\{(X_{i},A_{i})\}_{i \in\NN}$ such that $d_\GH((X_{i},A_{i}),(X_{k},A_{k}))<2\eps$ for all $i,k\in \NN$ and the conclusion follows as in \cite{herron}.
		
		As in \cite{herron}, we first get some $N\in \NN\cap (0,\nu(\eps/2)]$ and a subsequence $\{(X_{i},A_{i})\}_{i \in\NN}$ of $\{(X_{n},A_{n})\}_{n \in\NN}$ such that
		\[N(\eps/2,\BB_{2/\eps}(A_{i})) = N\] for all $i\in \NN$. In particular, there exist distinct points $\{x_{i1},\dots,x_{iN}\}\subset \BB_{2/\eps}(A_{i})$ such that 
		\[
		\BB_{2/\eps}(A_{i}) \subset \bigcup_{j=1}^{N}B_{\eps/2}(x_{ij})
		\] for all $i\in \NN$. Moreover, up to taking a subsequence and relabelling, we may assume there is some $1\leq k\leq N$ such that
		\[
		A_i \subset \bigcup_{j=1}^{k}B_{\eps/2}(x_{ij})
		\] and $A_i\cap B_{\eps/2}(x_{ij}) \neq \varnothing$ for all $1\leq j\leq k$ and for all $i\in \NN$.
		
		Now, by observing that 
		\[
		d_{X_i}(x_{im},x_{in}) \leq \diam(\BB_{2\eps}(A_i)) \leq C + \frac{4}{\eps}
		\] for all $i\in\NN$ and $1\leq m<n\leq N$ (where $C>0$ comes from the fact that $\mathcal{X}$ is uniformly bounded in the sense of pairs) we can apply the same argument as in \cite[Lemma 3.3]{herron} to extract a subsequence $\{(X_j,A_j)\}_{j\in \NN}$ of $\{(X_i,A_i)\}_{i\in \NN}$ (therefore a subsequence of $\{(X_n,A_n)\}_{n\in\NN}$) such that, for all $j,k\in \NN$ and all $1\leq m<n\leq N$,
		\[
		|d_{X_j}(x_{jm},x_{jn})-d_{X_k}(x_{km},x_{kn})|<\eps/2.
		\] Applying Lemma \ref{lem:GH-with-nets}, we get that $d_{\GH}((X_j,A_j),(X_k,A_k))<2\eps$.
	\end{proof}

	\section{Convergence of tuples}
	\label{sec:tuples}
	It is easy to extend the definitions and results in the previous section to the more general setting of tuples of metric spaces. 
	
	\begin{definition}
		A \emph{metric tuple} consists of a metric space $X$ and a finite nested sequence of subsets 
		\[
		X \supseteq X^N \supseteq X^{N-1} \supseteq \dots \supseteq X^1\supsetneq \varnothing, 
		\] 
		where each $X^i$ is a closed subset of $X$. We say that $(X,X^N,\dots, X^1)$ is an \emph{extended metric tuple} if the metric space $X$ is an extended metric space. We denote such metric tuple by $(X,X^N,\dots,X^1)$.
	\end{definition}
	
	First, we have the analogous notion of Gromov--Hausdorff convergence for tuples.
	
	\begin{definition}\label{def:Gromov-Hausdorff-CGGGMS-tuples} 
		A sequence $\{(X_i,X^N_i,\dots,X^1_i)\}_{i\in \NN}$ of metric tuples \emph{converges in the Gromov--Hausdorff topology} to a metric tuple $(X,X^N,\dots, X^1)$ if there exist sequences $\{\eps_i\}_{i\in\NN}$ and $\{R_i\}_{i\in\NN}$ of positive numbers with $\eps_i\searrow 0$, $R_i \nearrow\infty$, and $\eps_i$-approximations from $\BB_{R_i}(X^{N}_i)$ to
		$\BB_{R_i}(X^{N})$ for each $i\in\NN$, i.e. maps $\phi_i \colon \BB_{R_i}(X^N_i)\to X$ satisfying the following three conditions:
		\begin{enumerate}
			\item $\left| d_{X_i}(x,y)-d_X(\phi_i(x),\phi_i(y) \right|\leq\eps_i$ for any $x,y\in \BB_{R_i}(X^{N}_i)$;
			\item $d^{d_{X}}_{\H}(\phi_i(X^k_i),X^k)\leq\eps_i$ for $k\in \{1,\dots, N\}$;
			\item $\BB_{R_i}(X^k)\subset \BB_{\eps_i}(\phi_i(\BB_{R_i}(X^k_i)))$.
		\end{enumerate}
		We will denote the Gromov--Hausdorff convergence of metric tuples by 
		\[
		(X_i,X_i^N,\dots,X_i^1)\toGH (X,X^N,\dots,X^1).
		\]
	\end{definition}
	
	We can also define the Hausdorff distance between metric tuples that lie in the same metric space. 
	\begin{definition}
		Let $(Z,\delta)$ be a metric space, $X,Y \subset Z$ bounded subsets and tuples $(X,X^N,\dots,X^1)$, $(Y,Y^N,\dots,Y^1)$. The \emph{tuple Hausdorff distance} of $(X,X^N,\dots,X^1)$ and $(Y,Y^N,\dots,Y^1)$ is given by 
		\[
		d^\delta_{\H} ((X,X^N,\dots,X^1),(Y,Y^N,\dots,Y^1)) := d^\delta_\H (X,Y) + \sum_{k=1}^{N} d^\delta_\H(X^{k},Y^{k})
		\]
	\end{definition}
	We can now define the Gromov--Hausdorff distance of metric tuples in the compact case, just as we did for metric pairs.
	\begin{definition}\label{def:Gromov-Hausdorff tuples-Jansen}
		The \emph{tuple Gromov--Hausdorff distance} between two compact metric tuples $(X,X^N,\dots,X^1)$ and $(Y,Y^N,\dots,Y^1)$ is defined as
		\[
		d_{\GH} ((X,X^N,\dots,X^1),(Y,Y^N,\dots,Y^1)) := \inf\{d^\delta_\H ((X,X^N,\dots,X^1),(Y,Y^N,\dots,Y^1))\},
		\] where the infimum is taken over all admissible metrics $\delta$ on $X \sqcup Y$.
	\end{definition}
	
	For the non-compact case, we proceed as in the case of metric pairs by considering a suitable notion of admissible metrics in the disjoint union of metric tuples.
	\begin{definition}
		Given $\eps>0$ and metric tuples $(X,X^N,\dots,X^1)$, $(Y,Y^N,\dots,Y^1)$, an admissible distance function $\delta$ on $X\sqcup Y$ is \emph{$(\eps; (X^N,\dots,X^1),(Y^N,\dots, Y^1)$)-admissible} provided
		\[
		d^\delta_\H(X^i,Y^i)<\eps, \quad \BB_{1/\eps}^{\delta}(X^N)\subset B^{\delta}_\eps(Y), \quad \BB^{\delta}_{1/\eps}(Y^N)\subset B^{\delta}_\eps(X)
		\]
		for $i\in\{1,\dots,N\}.$
	\end{definition}
	
	\begin{definition}
		The \emph{tuple Gromov--Hausdorff distance} between not necessarily compact tuples $(X,X^N,\dots,X^1)$, $(Y,Y^N,\dots,Y^1)$ is given by
		\[
		d_{\GH}((X,X^N,\dots,X^1),(Y,Y^N,\dots,Y^1)) = \min\left\{\frac{1}{2},\widetilde{d}_{\GH}((X,X^N,\dots,X^1),(Y,Y^N,\dots,Y^1))\right\},
		\]
		where $\widetilde{d}_{\GH}((X,X^N,\dots,X^1),(Y,Y^N,\dots,Y^1))$ is the infimum of the set
		\[\{\eps>0:\text{there exists}\ \text{a $(\eps;(X^N,\dots,X^1),(Y^N,\dots,Y^1))$-admissible distance $\delta$ on }X\sqcup Y\}.\]
	\end{definition}
	
	\begin{remark}
		Also in this case, both definitions of tuple Gromov--Hausdorff distance induce the same topology on the space of compact tuples.
	\end{remark}

	Just as in the case of metric pairs, we get some equivalences between the previous definitions.
	
	\begin{proposition}
		Let $\{(X_i,X^N_i,\dots,X^1_i)\}_{i\in\NN}$ and $(X,X^N,\dots,X^1)$ proper metric tuples. Then:
		\begin{enumerate}
			\item The convergence $(X_i,X^N_i,\dots,X^1_i)\toGH (X,X^N,\dots,X^1)$ is equivalent to the condition  
			\[
			d_\GH((X_i,X^N_i,\dots,X^1_i),(X,X^N,\dots,X^1))\to 0.
			\]
			\item If $(X,X^N,\dots,X^1)$ and $\{(X_i,X^N_i,\dots,X^1_i)\}_{i\in \NN}$ are compact, \[(X_i,X^N_i,\dots,X^1_i)\toGH (X,X^N,\dots,X^1)\] is equivalent to the condition
			\begin{equation}\label{eq:convergence-jansen-tuples}
				d_\GH(X_i,X^N_i,\dots,X^1_i),(X,X^N,\dots,X^1))\to 0.
			\end{equation}
			In general, 
			\begin{equation}\label{eq:convergence-jansen-proper-tuples}
				d_\GH(\BB_r(X^N_i),X^N_i,\dots,X^1_i),(\BB_r(X^N),X^N,\dots,X^1))\to 0 \quad \text{for all }r>0  
			\end{equation}
			implies $(X_i,X^N_i,\dots,X^1_i)\toGH  (X,X^N,\dots,X^1)$. If in addition $\{X_i\}_{i\in \NN}$ and $X$ are length spaces, then the converse also holds.
		\end{enumerate}
	\end{proposition}

	We can also have versions of the Theorems \ref{thm:embedding}, \ref{thm:completeness} and \ref{thm:precompactness} in the setting of metric tuples. We omit the proofs of these results since they are completely analogous to those in the case of metric pairs.
	\begin{theorem}
		Let $\{(X_i,X^N_i,\dots X^1_i)\}_{i\in\NN}$ be a sequence of proper metric tuples. Suppose that
		\[
		\sum_{i=1}^{\infty} d_{\GH}((X_i,X^N_i,\dots,X^1_i),(X_{i+1},X^N_{i+1},\dots,X^1_{i+1})) < \infty.
		\]
		Then there exists a non-complete locally complete metric space $Y$ and a metric tuple $(Z,Z^N,\dots,Z^1)$, where $Z=\partial Y$, with the following properties:
		\begin{enumerate}
			\item for each $i$ the space $X_i$ naturally isometrically embeds into $Y$.
			\item the space $\overline{Y}=Y\cup Z$ is proper.
			\item $(X_i,X^N_i,\dots,X^1_i)\toGH (Z,Z^N,\dots,Z^1)$.
			\item For all $R>0$ there are $R_i>R$ such that $R_i\to R$ and, for all $j\in\{1,\dots,N\}$,  
			\[
			\BB_{R_i}(W^j)\cap X_i \xrightarrow[]{\H} \BB_R(W^j)\cap Z.
			\]
			Moreover, if all $X_i$ are length, then $Z$ is a geodesic, and in this setting,
			\item For all $R>0$ and all $j\in\{1,\dots,N\}$, 
			\[
			\BB_R(X^j_i)\cap X_i\xrightarrow{\H} \BB_R(W^j)\cap Z.
			\]
		\end{enumerate}
	\end{theorem}
	
	\begin{theorem}
		Let $\GHTuple_N$ denote the collection of all isometry classes of proper metric metric tuples $(X,X^N,\dots,X^1)$. Then $(\GHTuple_N,d_{\GH})$ is a complete metric space.
	\end{theorem}

	\begin{theorem}
		For any collection $\mathcal{X}$ of (isometry classes of) proper metric tuples such that $\{(X,X^N)\}_{X\in \mathcal{X}}$ is uniformly bounded in the sense of pairs, the following are equivalent:
		\begin{enumerate}
			\item $\mathcal{X}$ is precompact with respect to $d_\GH$.
			\item There exists $\pi\colon (0,\infty)\to (0,\infty)$ such that for all $\eps>0$,
			\[
			P(\eps, \BB_{1/\eps}(X^N))\leq \pi(\eps)
			\]
			for all $(X,X^N,\dots, X^1)\in \mathcal{X}$.
			\item There exists $\nu\colon (0,\infty)\to (0,\infty)$ such that for all $\eps>0$,
			\[
			N(\eps, \BB_{1/\eps}(X^N))\leq \nu(\eps)
			\]
			for all $(X,X^N\dots, X^1)\in \mathcal{X}$.
		\end{enumerate}
	\end{theorem}
	
	\section{Applications}
	\label{sec:applications}
	
	In this section we apply our framework in different settings where metric pairs and metric tuples naturally emerge. Our first two applications pertain to the theory of Alexandrov spaces with curvature bounded below. We refer the reader to \cite{burago} for an introduction to this subject.

	We use the following notation throughout the section. Given $n\in \NN$, $k\in \RR$, $v>0$ and $D>0$, define
	\[
	\calM(n,k,D) = \left\{X:
	\begin{matrix}
		X\ \text{is an $n$-dimensional closed}\\ 
		\text{Riemannian manifold with}\\
		\diam(X)\leq D\ \text{and}\ \sec_{X}\geq k
	\end{matrix}
	\right\},
	\]
	\[
	\calN(n,k,D,v) = \left\{X:
	\begin{matrix}
		X\ \text{is a closed $n$-dimensional}\\ 
		\text{Riemannian manifold with}\\
		\diam(X)\leq D,\ \vol(X)>v\ \text{and}\ \left|\sec_{X}\right|\leq k
	\end{matrix}
	\right\}\ \text{(for $k>0$)},
	\]
	\[
	\calA(n,k,D) = \left\{X:
	\begin{matrix}
		X\ \text{is an $n$-dimensional closed}\\ 
		\text{Alexandrov space with}\\
		\diam(X)\leq D\ \text{and}\ \curv(X)\geq k
	\end{matrix}
	\right\}.
	\]
	
	For the first application, let us recall what extremal subsets of Alexandrov spaces are. Namely, given an Alexandrov space $X$, we say that a closed subset $E\subset X$ is \emph{extremal} if for any $q\in X\setminus E$, whenever $d_X(q,\cdot)|_E$ has a local minimum at $p\in E$, $p$ is a critical point of $d_X(q,\cdot)$ (see \cite{petrunin} for details). Theorem \ref{thm:extremal} is a natural extension of \cite[Lemma 4.1.3]{petrunin} to tuples of extremal sets.

	\begin{proof}[Proof of Theorem \ref{thm:extremal}]
		In \cite[Lemma 4.1.3]{petrunin}, the author proved that if $E_i\subset X_i$ is a sequence of extremal sets such that $E_i\to E$ for some set $E\subset X$ under the convergence $X_i\toGH X$ (i.e. the sets $\phi_i(E_i)$ converge to $E$ with respect to the Hausdorff distance in $X$ for some fixed $\eps_i$-approximations $\phi_i\colon X_i\to X$ with $\eps_i\searrow 0$), then $E$ is an extremal set in $X$. The result follows by applying this result and Definition \ref{def:Gromov-Hausdorff tuples-Jansen}.
	\end{proof}

	For the next application, we recall there is a  notion of \emph{equivariant Gromov--Hausdorff convergence} for isometric actions on metric spaces which was introduced by K. Fukaya in \cite{fukaya1}. Roughly speaking, we say that the actions by isometries $\{(X_i,G_i)\}_{i\in\NN}$ converge in the equivariant Gromov--Hausdorff sense to the action by isometries $(X,G)$, which we denote by $(X_i,G_i) \xrightarrow{eGH} (X,G)$, if there exist Gromov--Hausdorff approximations between the $X_i$ and $X$ that get closer and closer to be equivariant isometries as $i$ goes to infinity.
	
	Theorem \ref{thm:equivariant-convergence} is a consequence of K. Fukaya results in \cite{fukaya1} and \cite{fukaya2} which can be formulated using our framework. This might be thought of as an example of the conclusion of the previous Theorem \ref{thm:extremal}, since $X^H/G$ is an extremal subset in $X/G$ whenever $X\in \calA(n,k,D)$, $G$ is a group of isometries of $X$ and $H$ is a subgroup of $G$ \cite{petrunin}.

	\begin{proof}[Proof of Theorem \ref{thm:equivariant-convergence}]
		Observe that due to \cite[Theorem 2-1]{fukaya1}, we know that $(X_i,G)\xrightarrow[]{eGH} (X,G)$ implies that $X_i/G\toGH X/G$. Since the corresponding actions restrict to each $X_i^{H_j}$ and $X^{H_j}$, we can also apply \cite[Theorem 2-1]{fukaya1} to get $X_i^{H_j}\to X^{H_j}$, which implies the result.
		
		For the second part, by \cite[Theorem 6.9]{fukaya2}, since we have a non-collapsing convergence, we know that for sufficiently large $i\in\NN$ there exists an equivariant homeomorphism $f_i\colon (X_i,G)\to (X,G)$ which is also a Gromov--Hausdorff approximation. In particular, $f_i|\colon X_i^{H_j}\to X^{H_j}$ is also a homeomorphism and a Gromov--Hausdorff approximation. This clearly implies the first convergence. 
		
	\end{proof}

	The last application is related to the setting of stratified spaces. These are singular metric spaces where singularities are given by manifolds that are attached in a reasonable way. Examples of this notion are given by Riemannian singular foliations and quotient spaces of isometric actions on closed manifolds. There are many equivalent definitions of stratified spaces and we refer the reader to \cite{bertrand21,pflaum01} for an introduction to the analytic and geometric theory of stratified spaces.

	We use the following notation: given $n\in \NN$, $k\in \RR$, $D>0$ and $v>0$, define
	\[
	\calS(n,k,D,v) = \left\{(X^n,\dots,X^0)\in \GHTuple_n:
	\begin{array}{c}
		X^n=\bigsqcup_{i=0}^{n} \Sigma^i\ \text{is an $n$-dimensional closed stratified}\\ 
		\text{manifold, where $\Sigma^i$ is the $i$-dimensional stratum,}\\ 
		\text{$X^{i} = \bigsqcup_{j=0}^{i} \Sigma^j$, $X^n\in \calA(n,k,D)$, and \ $\vol^n(X^n)\geq v$}
	\end{array}
	\right\}
	\]
	where $\vol^n$ denotes the $n$-dimensional Hausdorff measure. Moreover, we define
	\[
	\calS^e(n,k,D,v) = \left\{(X^n,\dots,X^0)\in \calS(n,k,D,v):
	\begin{array}{c}
		\text{$\overline{\Sigma^i}$ is an extremal set of $X^n$,}\\ 
		\text{whenever $\Sigma^i= X^i\setminus X^{i+1}$ is non-empty}
	\end{array}
	\right\}.
	\]
	Observe that for the previous definitions we are sticking to the notation in \cite{bertrand21}.
	
	\begin{proof}[Proof of Theorem \ref{thm:stratified-spaces}]
		Let $\{(X^n_i,\dots, X^0_i)\}_{i\in\NN}$ be a sequence of metric tuples in $\calS(n,k,D,v)$, with $X^n_i = \bigsqcup_{j=0}^{n}\Sigma^j_i$. In particular, $\{X^n_i\}_{i\in \NN}\subset \calA(n,k,D,v)$, which is known to be precompact with respect to the usual Gromov--Hausdorff convergence (see \cite[Proposition 10.7.3]{burago}). Therefore, up to passing to a subsequence, there exists $X^n\in \calA(n,k,D)$ such that $X^n_i\toGH X^n$ (see \cite[Corollaries 10.8.25, 10.10.11]{burago}). Moreover, by considering Gromov--Hausdorff approximations $f_i\colon X^n_i\to X^n$, which are homeomorphisms for sufficiently large $i\in\NN$ due to Perelman's Stability Theorem \cite{kapovitch,perelman}, and due to the compactness of $X^n$, we can assume (again, possibly after passing to a subsequence) that $X^j_i\toGH X^j$ for some closed set $X^j\subset X^n$, for each $j\in \{0,\dots,n\}$. Then, in particular, 
		\[
		(X^n_i,\dots,X^0_i)\toGH (X^n,\dots, X^0).
		\]
		Therefore, $\calS(n,k,D,v)$ is precompact as claimed. 
		
		Now, if we assume that there are infinitely many relative homeomorphism types in $\calS^e(n,k,D,v)$, we can consider a sequence $\{(X^n_i,\dots,X^0_i)\}_{i\in\NN}$ in $\calS^e(n,k,D,v)$ such that 
		$(X_i^n,\dots, X_i^0)$ and $(X_j^n,\dots, X_j^0)$ are not homeomorphic as tuples (i.e. there is no homeomorphism $X_i^n\to X_j^n$ preserving the elements in the tuple) for any $i\neq j$, $i,j\in \NN$. However, by the precompactness of $\calS(n,k,D,v)$, up to passing to a subsequence, there exists some $(X^n,\dots,X^0) \in \GHTuple_n$ such that 
		\[
		(X_i^n,\dots, X_i^0)\toGH (X^n,\dots,X^0).
		\]
		Moreover, since $\vol^n(X^n) \geq v>0$, and $X^n_i\toGH X^n$, we know that $X^n\in \calA(n,k,D)$ (again, see \cite[Corollaries 10.8.25, 10.10.11]{burago}). Furthermore, due to \cite[Lemma 4.1.3]{petrunin} and the fact that $(X^n_i,\dots,X^0_i)\toGH (X^n,\dots, X^0)$, we get that $\overline{\Sigma}^j_i\to \overline{\Sigma}^j$ and $\overline{\Sigma}^j$ is an extremal set in $X^n$, for each $j\in\{0,\dots,n\}$. By \cite[Theorem 4.3]{HS17}, there exists a homeomorphism $(X^n_i,\dots,X^0_i)\to (X^n,\dots,X^0)$ for sufficiently large $i\in\NN$, which is a contradiction.
	\end{proof}

	\bibliographystyle{plainurl} 
	\bibliography{bibliografia}
	
\end{document}